\documentclass[12pt]{article}
\usepackage[numbers]{natbib}
\usepackage{amsmath,amsthm,amsfonts,amssymb,enumerate,calc,graphicx,iwona,float}
\usepackage[colorlinks=true,citecolor=black,linkcolor=black,urlcolor=blue]{hyperref}
\usepackage[margin=28mm]{geometry}
\usepackage[noabbrev,capitalise]{cleveref}
\usepackage{microtype}
\renewcommand{\baselinestretch}{1.15}
\usepackage[hang]{footmisc} 

\renewcommand{\thefootnote}{\fnsymbol{footnote}}	
\allowdisplaybreaks
\newcommand\DateFootnote{
\begingroup
\renewcommand\thefootnote{}
\footnote{\today}
\setcounter{footnote}{0}
\vspace*{-3ex}
\endgroup}

\makeatletter
\renewcommand\section{\@startsection {section}{1}{\z@}%
                                   {-3ex \@plus -1ex \@minus -.2ex}%
                                   {2ex \@plus.2ex}%
                                   {\normalfont\large\bfseries}}
\renewcommand\subsection{\@startsection{subsection}{2}{\z@}%
                                     {-2.5ex\@plus -1ex \@minus -.2ex}%
                                     {1.5ex \@plus .2ex}%
                                     {\normalfont\normalsize\bfseries}}
\renewcommand\subsubsection{\@startsection{subsubsection}{3}{\z@}%
                                     {-2ex\@plus -1ex \@minus -.2ex}%
                                     {1ex \@plus .2ex}%
                                     {\normalfont\normalsize\bfseries}}
 \renewcommand\paragraph{\@startsection{paragraph}{4}{\z@}%
                                    {1.5ex \@plus.5ex \@minus.2ex}%
                                    {-1em}%
                                    {\normalfont\normalsize\bfseries}}
\renewcommand\subparagraph{\@startsection{subparagraph}{5}{\parindent}%
                                       {1.5ex \@plus.5ex \@minus .2ex}%
                                       {-1em}%
                                      {\normalfont\normalsize\bfseries}}
\makeatother

\renewcommand{\thefootnote}{\fnsymbol{footnote}}	
\newcommand{\arXiv}[1]{arXiv:\,\href{http://arxiv.org/abs/#1}{#1}}

\newcommand{\ceil}[1]{\lceil{#1}\rceil}
\newcommand{\floor}[1]{\lfloor{#1}\rfloor}
\newcommand{\half}{\ensuremath{\protect\tfrac{1}{2}}}
\renewcommand{\geq}{\geqslant}
\renewcommand{\leq}{\leqslant}

\DeclareMathOperator{\dist}{dist}

\DeclareMathOperator{\tw}{tw}
\DeclareMathOperator{\pw}{pw}

\theoremstyle{plain}
\newtheorem{theorem}{Theorem}
\newtheorem{lemma}[theorem]{Lemma}
\newtheorem{corollary}[theorem]{Corollary}

\theoremstyle{definition}

\begin{document}
{\Large\bfseries\boldmath\scshape Structure of Graphs with Locally Restricted Crossings\footnotemark[0]}


\DateFootnote

Vida Dujmovi{\'c}\,\footnotemark[2]
\quad 
David Eppstein\,\footnotemark[3] 
\quad  
David~R.~Wood\,\footnotemark[4]

\footnotetext[0]{A preliminary version of this paper entitled ``Genus, treewidth, and local crossing number'' was published in \emph{Proc. 23rd International Symposium on Graph Drawing and Network Visualization 2015}, Lecture Notes in Computer Science 9411:87--98, Springer, 2015.}

\footnotetext[2]{School of Computer Science and Electrical Engineering, University of Ottawa, 
Ottawa, Canada (\texttt{vida.dujmovic@uottawa.ca}). Supported by NSERC and the Ministry of Research and Innovation, Government of Ontario, Canada.}

\footnotetext[3]{Department of Computer Science, University of California, Irvine, California, USA (\texttt{eppstein@uci.edu}). Supported in part by NSF grant  CCF-1228639.}

\footnotetext[4]{School of Mathematical Sciences, Monash University, Melbourne, Australia (\texttt{david.wood@monash.edu}). Supported by the Australian Research Council.}

\emph{Abstract.} We consider relations between the size, treewidth, and local crossing number (maximum number of crossings per edge) of graphs embedded on topological surfaces. We show that an $n$-vertex graph embedded on a surface of genus $g$ with at most $k$ crossings per edge has treewidth $O(\sqrt{(g+1)(k+1)n})$ and layered treewidth $O((g+1)k)$, and that these bounds are tight up to a constant factor.  As a special case, the $k$-planar graphs with $n$ vertices have treewidth $O(\sqrt{(k+1)n})$ and layered treewidth $O(k+1)$, which are tight bounds that improve a previously known $O((k+1)^{3/4}n^{1/2})$ treewidth bound. Analogous results are proved for  map graphs defined with respect to any surface. Finally, we show that for $g<m$, every $m$-edge graph can be embedded on a surface of genus~$g$ with $O((m/(g+1))\log^2 g)$ crossings per edge, which is tight to a polylogarithmic factor.


\section{Introduction}
\label{Intro}

\renewcommand{\thefootnote}{\arabic{footnote}}

This paper studies the structure of graph classes defined by drawings on surfaces in which the crossings are locally restricted in some way. 

The first such example that we consider are the $k$-planar graphs. A graph is \emph{$k$-planar} if it can be drawn in the plane with at most $k$ crossings on each edge \citep{PachToth-Comb97}. The \emph{local crossing number} of the graph is the minimum $k$ for which it is $k$-planar~\cite[pages 51--53]{Schaefer14}. An important example is the  $p\times q\times r$ grid graph, with vertex set $[p]\times[q]\times[r]$ and all edges of the form $(x,y,z)(x+1,y,z)$ or $(x,y,z)(x,y+1,z)$ or $(x,y,z)(x,y,z+1)$. A suitable linear projection from the natural three-dimensional embedding of this graph to the plane gives a $(r-1)$-planar drawing, as  illustrated in \cref{GridGraph}.

\begin{figure}[t]
\centering
\includegraphics[width=0.7\textwidth]{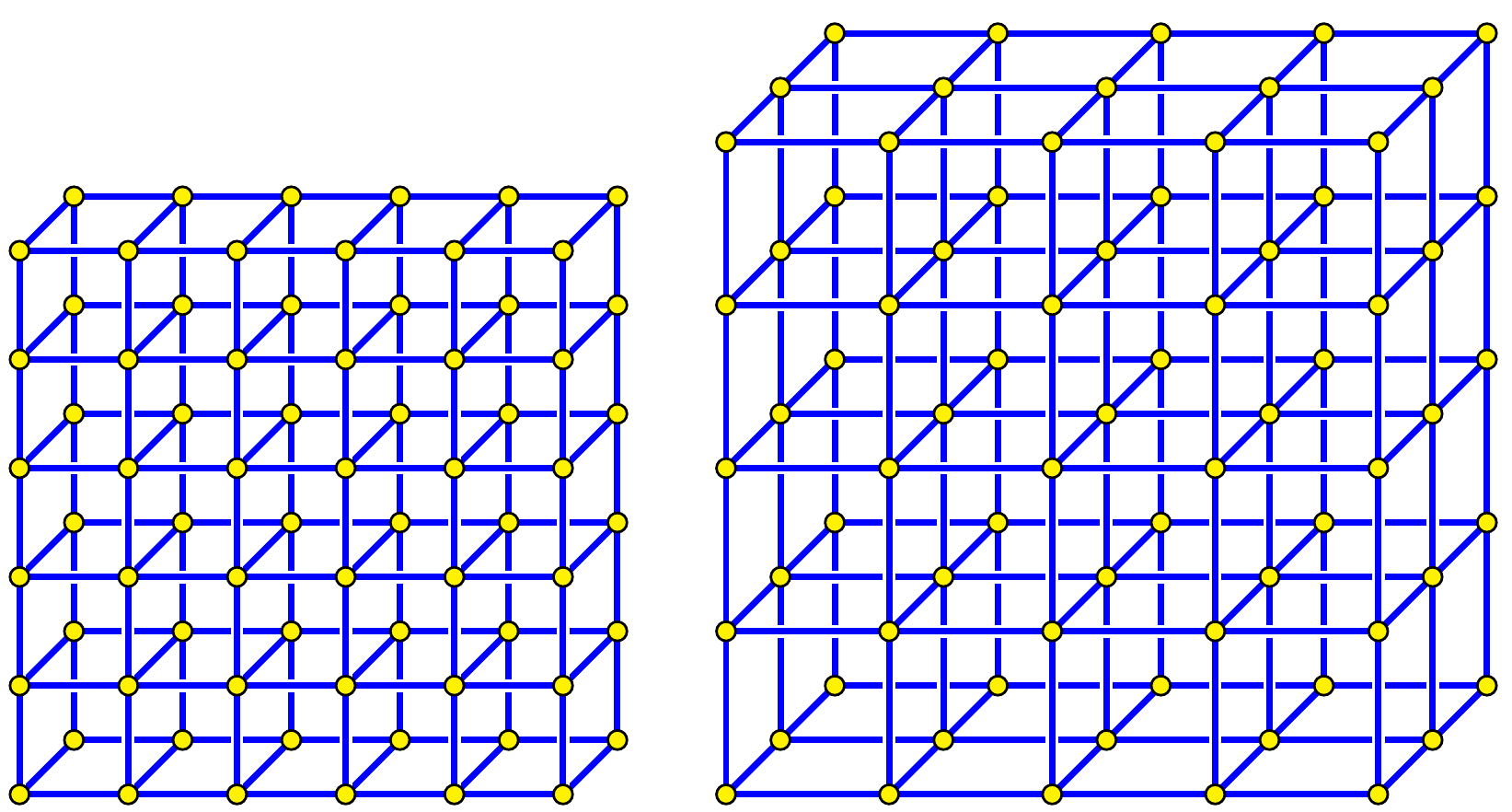}
\caption{The $p\times q\times r$ grid graph is $(r-1)$-planar. }
\label{GridGraph}
\end{figure}

The main way that we describe the structure of a graph is through its treewidth, which is a parameter that measures how similar a graph is to a tree. It is a key measure of the complexity of a graph and is of fundamental importance in algorithmic graph theory and structural graph theory, especially in Robertson and Seymour's graph minors project. See \cref{Background} for a detailed definition of treewidth. 

Treewidth is closely related to the size of a smallest \emph{separator}, a set of vertices whose removal splits the graph into connected components each with at most half the vertices. Graphs of low treewidth necessarily have small separators, and graphs in which every subgraph has a small separator have low treewidth \citep{DN14,Reed97}. For example, the Lipton-Tarjan separator theorem, which says that every $n$-vertex planar graph has a separator of order $O(\sqrt{n})$, can be reformulated as every $n$-vertex planar graph has treewidth $O(\sqrt{n})$. Most of our results provide $O(\sqrt{n})$ bounds on the treewidth of particular classes of graphs that generalise planarity. In this sense, our results are generalisations of the Lipton-Tarjan separator theorem, and analogous results for other surfaces. 

The starting point for our work is the following question: what is the maximum treewidth of $k$-planar graphs on $n$ vertices? \citet{GB07} studied this question and proved an upper bound of $O(k^{3/4}n^{1/2})$. We improve this and give the following tight bound:

\begin{theorem}
\label{kPlanarTreewidth}
The maximum treewidth of $k$-planar $n$-vertex graphs is
\[ \Theta\left(\min\left\{n,\sqrt{(k+1)n}\right\}\right). \]
\end{theorem}

More generally, a graph is \emph{$(g,k)$-planar} if it can be drawn in a surface of Euler genus at most $g$ with at most $k$ crossings on each edge\footnote{The \emph{Euler genus} of an orientable surface with $h$ handles is $2h$. The \emph{Euler genus} of  a non-orientable surface with $c$ cross-caps is $c$. The \emph{Euler genus} of a graph $G$ is the minimum Euler genus of a surface in which $G$ embeds (with no crossings).}. For instance, \citet{GJS68} investigated the local crossing number of toroidal embeddings---in this notation, the $(2,k)$-planar graphs. We again determine an optimal bound on the treewidth of such graphs. 

\begin{theorem}
\label{gkPlanarTreewidth}
The maximum treewidth of $(g,k)$-planar $n$-vertex graphs is
\[ \Theta\left(\min\left\{n,\sqrt{(g+1)(k+1)n}\right\}\right). \]
\end{theorem}

In both these theorems, the $k=0$ case (with no crossings) is well known~\citep{GHT-JAlg84}. 

Our second contribution is to study the $(g,k)$-planarity of graphs as a function of their number of edges. For (global) crossing number, it is known that a graph with $n$ vertices and $m$ edges  drawn on a surface of genus~$g$ (sufficiently small with respect to $m$) may require $\Omega(\min\{m^2/g,m^2/n\})$ crossings, and it can be drawn with $O((m^2\log^2 g)/g)$ crossings~\citep{SSSV96}. In particular, the lower bound implies that some graphs require $\Omega(m/g)$ crossings per edge on average, and therefore also in the worst case.
We prove a nearly-matching upper bound which implies the above-mentioned upper bound on the total number of crossings:

\begin{theorem}
\label{thm:few-crossings}
For every graph $G$ with $m$ edges, for every integer $g\geq 1$, there is a drawing of $G$ in the orientable surface with at most $g$ handles and with \[O\left(\frac{m\log^2 g}{g}\right)\]
crossings per edge.
\end{theorem}

Our third contribution concerns map graphs, which are defined as follows. Start with a graph $G_0$ embedded in a surface of Euler genus $g$, with each face labelled a `nation' or a `lake', where each vertex of $G_0$ is incident with at most $d$ nations. Let $G$ be the graph whose vertices are the nations of $G_0$, where two vertices are adjacent in $G$ if the corresponding faces in $G_0$ share a vertex. Then $G$ is called a \emph{$(g,d)$-map graph}. A $(0,d)$-map graph is called a (plane) \emph{$d$-map graph}; such graphs have been extensively studied \citep{FLS-SODA12,Chen-JGT07,DFHT05,CGP02,Chen01}. It is easily seen that $(g,3)$-map graphs are precisely the graphs of Euler genus at most $g$ (which is well known in the $g=0$ case \citep{CGP02})\footnote{Let $G$ be a  graph embedded in a surface of Euler genus at most $g$.  Let $M(G)$ be the \emph{medial} graph of $G$. This graph has vertex set $E(G)$ where two vertices of $M(G)$ are adjacent whenever the corresponding edges in $G$ are consecutive in the cyclic ordering of edges incident to a common vertex in the embedding of $G$. Note that $M(G)$ embeds in the same surface as $G$, where each face of $M(G)$ corresponds to a vertex or a face of $G$. Label the faces of $M(G)$ that correspond to vertices of $G$ as nations, and label the faces of $M(G)$ that correspond to faces of $G$ as lakes. The vertex of $M(G)$ corresponding to an edge $vw$ of $G$ is incident to the nations corresponding to $v$ and $w$ (and is incident to no other nations). Thus $G$ is isomorphic to the map graph of $M(G)$, and $G$ is a $(g,2)$-map graph and thus a $(g,3)$-map graph. Conversely, it is clear that a $(g,3)$-map graph embeds in the same surface as the original graph.}. So $(g,d)$-map graphs provide a natural generalisation of graphs embedded in a surface. Note that $G$ may contain arbitrarily large cliques even in the $g=0$ case, since if a vertex of $H$ is incident with $d$ nations then $G$ contains $K_d$. 

If $G$ is the map graph associated with an embedded graph $H$, then consider the natural drawing of $G$ in which each vertex  of $G$ is positioned inside the corresponding nation, and each edge of $G$ is drawn as a curve through the corresponding vertex of $H$. If a vertex $v$ of $H$ is incident to $d$ nations, then each edge passing through $v$ is crossed by at most $\floor{\frac{d-2}{2}}\ceil{\frac{d-2}{2}}$ edges. Thus every $(g,d)$-map graph is $(g,\floor{\frac{d-2}{2}}\ceil{\frac{d-2}{2}})$-planar, and \cref{gkPlanarTreewidth} implies that every $(g,d)$-map graph has treewidth  $O(d\sqrt{(g+1)n})$. We improve on this result as follows.

\begin{theorem} 
\label{MapGraphTreewidthSummary}
The maximum treewidth of $(g,d)$-map graphs on $n$ vertices is
$$ \Theta\left(\min\left\{n,\sqrt{(g+1)(d+1)n}\right\}\right). $$
\end{theorem}

We prove our treewidth upper bounds by using the concept of \emph{layered treewidth}~\citep{DMW13}, which is of independent interest (see \cref{Background}). We prove matching lower bounds by finding $(g,k)$-planar graphs and $(g,d)$-map graphs without small separators and using the known relations between separator size and treewidth.

\section{Background and Discussion}
\label{Background}

For $\epsilon\in(0,1)$, a set $S$ of vertices in a graph $G$ is an \emph{$\epsilon$-separator} of $G$ if each component of $G-S$ has at most $\epsilon |V(G)|$ vertices. It is conventional to set $\epsilon=\frac{1}{2}$ or $\epsilon=\frac{2}{3}$ but the precise choice makes no difference to the asymptotic size of a separator.

Several results that follow depend on expanders; see \citep{HLW06} for a survey. The following folklore result provides a property of expanders that is the key to our applications. 

\begin{lemma}
\label{Expander}
For every $\epsilon\in(0,1)$ there exists $\beta>0$, such that for all
$k\geq 3$ and $n\geq k+1$ (such that $n$ is even if $k$ is odd), there exists a $k$-regular $n$-vertex graph $H$ (called an expander) in which every $\epsilon$-separator in $H$ has size at least $\beta n$.
\end{lemma}

A \emph{tree-decomposition} of  a graph $G$ is given by a tree $T$ whose nodes index a collection $(B_x\subseteq V(G):x\in V(T))$ of sets of vertices in $G$ called  \emph{bags}, such that:
\begin{itemize}
\item For every edge $vw$ of $G$, some bag $B_x$ contains both $v$ and $w$, and 
\item For every vertex $v$ of $G$, the set $\{x\in V(T):v\in B_x\}$ induces a non-empty (connected) subtree of $T$.
\end{itemize}
The \emph{width} of a tree-decomposition is $\max_x |B_x|-1$, and the \emph{treewidth} $\tw(G)$ of a graph $G$ is the minimum width of any tree decomposition of $G$. Path decompositions and pathwidth $\pw(G)$ are defined analogously, except that the underlying tree is required to be a path. Treewidth was introduced (with a different but equivalent definition) by \citet{Halin76} and tree decompositions were introduced by \citet{RS-GraphMinorsII-JAlg86} who proved:

\begin{lemma}[\citep{RS-GraphMinorsII-JAlg86}] 
\label{RS}
Every graph with treewidth $k$ has a $\frac12$-separator of size at most $k+1$. 
\end{lemma}

The notion of \emph{layered tree decompositions} is a key tool in proving our main theorems.  A \emph{layering} of a graph $G$ is a partition $(V_0,V_1,\dots,V_t)$ of $V(G)$ such that for every edge $vw\in E(G)$, if $v\in V_i$ and $w\in V_j$, then $|i-j|\leq 1$. Each set $V_i$ is called a \emph{layer}.  For example, for a vertex $r$ of a connected graph $G$, if $V_i$ is the set of vertices at distance $i$ from $r$, then $(V_0,V_1,\dots)$ is a layering of $G$, called the \emph{bfs layering} of $G$ starting from $r$. A \emph{bfs tree} of $G$ rooted at $r$ is a spanning tree  of $G$ such that for every vertex $v$ of $G$, the distance between $v$ and $r$ in $G$ equals the distance between $v$ and $r$ in $T$. Thus, if $v\in V_i$ then the $vr$-path in $T$ contains exactly one vertex from layer $V_j$ for $0\leq j\leq i$.

The \emph{layered width} of a tree-decomposition $(B_x:x\in V(T))$ of a graph $G$ is the minimum integer $\ell$ such that, for some layering $(V_0,V_1,\dots,V_t)$ of $G$, each bag $B_x$ contains at most $\ell$ vertices in each layer $V_i$. The \emph{layered treewidth} of a graph $G$ is the minimum layered width of a tree-decomposition of $G$. Note that if we only consider the trivial layering in which all vertices belong to one layer, then layered treewidth equals  treewidth plus 1.
\citet{DMW13} introduced layered treewidth\footnote{\citet{DMW13} introduced layered treewidth as a tool to prove upper bounds on the track-number, queue-number and volume of 3-dimensional straight-line grid drawings of graphs. In particular, based on earlier work in \citep{Duj15,DMW05}, they proved that every $n$-vertex graph with bounded layered treewidth has track-number $O(\log n)$, queue-number $O(\log n)$, and has a 3-dimensional straight-line grid  drawing with $O(n\log n)$ volume. All the theorems in this paper giving upper bounds on the layered treewidth of particular graph classes can be combined with the results in \citep{DMW13} to give results for track-number, queue-number, and 3-dimensional straight-line grid  drawings for the same graph class. Motivated by other applications, \citet{Shahrokhi13} independently introduced a definition equivalent to layered treewidth. Our results can also be combined with those of \citet{Shahrokhi13}; details omitted.}

\begin{theorem}[\citep{DMW13}]
\label{DMW} 
Every planar graph has layered treewidth at most $3$. More generally, 
every graph with Euler genus $g$ has layered treewidth at most $2g+3$.
\end{theorem}

Layered treewidth is related to local treewidth, which was first introduced by \citet{Epp-Algo-00} under the guise of the `treewidth-diameter' property. A graph class $\mathcal{G}$ has \emph{bounded local treewidth} if there is a function $f$ such that for every graph $G$ in $\mathcal{G}$, for every vertex $v$ of $G$ and for every integer $r\geq0$, the subgraph of $G$ induced by the vertices at distance at most $r$ from $v$ has treewidth at most $f(r)$; see \citep{Grohe-Comb03,DH-SJDM04,DH-SODA04,Epp-Algo-00}. If $f(r)$ is a linear or quadratic function, then  $\mathcal{G}$ has \emph{linear} or  \emph{quadratic} / \emph{local treewidth}. \citet{DMW13} observed that if every graph in some class $\mathcal{G}$ has layered treewidth at most $k$, then $\mathcal{G}$ has linear local treewidth with $f(r) \leq k(2r+1) -1$. They also proved the following converse result for minor-closed classes, where a graph $G$ is \emph{apex} if $G-v$ is planar for some vertex $v$. (Earlier, \citet{Epp-Algo-00} proved that (b) and (d) are equivalent, and  \citet{DH-SODA04} proved that (b) and (c) are equivalent.)

\begin{theorem}[\citep{DMW13,DH-SODA04,Epp-Algo-00}]
\label{MinorClosedLayered}
The following are equivalent for a minor-closed class $\mathcal{G}$ of graphs:
\begin{enumerate}[(a)]
\item $\mathcal{G}$ has bounded layered treewidth.
\item $\mathcal{G}$ has bounded local treewidth.
\item $\mathcal{G}$ has linear local treewidth.
\item $\mathcal{G}$ excludes some apex graph as a minor. 
\end{enumerate}
\end{theorem}

This result applies for neither $(g,k)$-planar graphs nor $(g,d)$-map graphs, since as we now show, these are non-minor-closed classes even for $g=0$, $k=1$ and $d=4$. For example, the $n\times n\times 2$ grid graph is 1-planar, and contracting the $i$-th row in the front grid with the $i$-th column in the back grid gives a $K_n$ minor. Thus 1-planar graphs may contain arbitrarily large complete graph minors.  Similarly, we now construct $(0,4)$-map graphs with arbitrarily large complete graph minors. Let $H_n$ be the $(2n+1)\times(2n+1)$ grid graph in which each internal face is a nation, and the outer face is a lake. Let $G_n$ be the map graph of $H_n$. Since $H_n$ is planar with maximum  degree 4, $G_n$ is a $(0,4)$-map graph. Observe that $G_n$ is the $2n\times2n$ grid graph with both diagonals across each face. Say $V(G_n)=[1,2n]^2$. 
For $i\in[1,n]$, let $R_i$ be the zig-zag path 
$(1,2i-1)(2,2i)(3,2i-1),(4,2i),\dots,(2n-2,2i),(2n-1,2i-1),(2n,2i)$ in $G_n$, 
let $C_i$ be the zig-zag path 
$(2i,1)(2i-1,2)(2i,3),(2i-1,4),\dots,(2i-1,2n-2),(2i,2n-1),(2i-1,2n)$ in $G_n$, 
and let $X_i$ be the subgraph $R_i\cup C_i$. 
Then $X_i$ is connected since $(2i-1,2i-1)\in R_i$ is adjacent to $(2i-1,2i)\in C_i$. 
Note that the sum of the coordinates of each vertex in $R_i$ is even, and the sum of the coordinates of each vertex in $C_j$ is odd. Thus $R_i\cap C_j=\emptyset$ for all $i,j\in[n]$. Clearly $R_i\cap R_j=\emptyset$ and $C_i\cap C_j=\emptyset$ for distinct $i,j\in[1,n]$. Thus $X_i\cap X_j=\emptyset$ for distinct $i,j\in[1,n]$. 
Now, $X_i$ is adjacent to $X_j$ since $(2j,2i)\in R_i$ is adjacent to $(2j-1,2i)\in C_j$. 
Thus $X_1,\dots,X_n$ are the branch sets of a $K_n$ minor in $G_n$. This example shows that $(0,4)$-map graphs may contain arbitrarily large complete graph minors. 

Sergey Norin established the following connection between layered treewidth and treewidth.

\begin{lemma}[Norin; see~\citep{DMW13}]
\label{Norine}
Every $n$-vertex graph with layered treewidth $k$ has treewidth at most $2\sqrt{kn}-1$.
\end{lemma}

To prove all the $O(\sqrt{n})$ treewidth bounds introduced in \cref{Intro}, we first establish a tight upper bound on the layered treewidth, and then apply \cref{Norine}. One conclusion, therefore, of this paper is that layered treewidth is a useful parameter when studying non-minor-closed graph classes (which is a research direction suggested by \citet{DMW13}). In general, layered treewidth is an interesting measure of the structural complexity of a graph in its own right. 

%

We now show that bounded local treewidth does not imply bounded layered treewidth (and thus \cref{MinorClosedLayered} does not necessarily hold in non-minor-closed classes). First note that a graph with maximum degree $\Delta$ contains  $O((\Delta-1)^r)$ vertices at distance at most $r$ from a fixed vertex (the \emph{Moore} bound). Thus graphs with maximum degree $\Delta$ have bounded local treewidth. Let $G_n$ be the $n\times n\times n$ grid graph, which has maximum degree 6. Thus  $\{G_n:n\in\mathbb{N}\}$ has bounded local treewidth. Moreover, the subgraph of $G_n$ induced by the vertices at distance at most $r$ from a fixed vertex is a subgraph of $G_{2r}$, which is easily seen to have treewidth $O(r^2)$. Thus  $\{G_n:n\in\mathbb{N}\}$ has quadratic local treewidth. By \cref{TreewidthThreeDimGrid} in \cref{gkPlanarGraphs} below, $\tw(G_n)\geq \frac{1}{6}n^2$. If $G_n$ has layered treewidth $k$, then $\tw(G_n)\leq 2\sqrt{k n^3}$ by  \cref{Norine}. Thus $\frac{1}{6}n^2 \leq 2\sqrt{k n^3}$, which implies that  $k\geq \frac{1}{144} n$, and $\{G_n:n\in\mathbb{N}\}$ has unbounded layered treewidth. 

%

We conclude this section by mentioning some negative results. \citet{DSW16} constructed an infinite family of expander graphs that have (geometric) thickness 2, have  3-page book embeddings, have 2-queue layouts, and have 4-track layouts. By \cref{Expander}, \cref{RS}  and \cref{Norine}, such graphs have treewidth $\Omega(n)$ and layered treewidth $\Omega(n)$. This means that our results cannot be extended to bounded thickness, bounded page number, bounded queue number, or bounded track number graphs. 
 
\section{\boldmath $k$-Planar Graphs}
\label{kPlanarGraphs}

The following theorem is our first contribution. 

\begin{theorem}
\label{kPlanarLayeredTreewidth}
Every $k$-planar graph $G$ has layered treewidth at most $6(k+1)$.
\end{theorem}

\begin{proof}
Draw $G$ in the plane with at most $k$ crossings per edge, and arbitrarily orient each edge of $G$. Let $G'$ be the graph obtained from $G$ by replacing each crossing by a new degree-4 vertex. Then $G'$ is planar. By \cref{DMW}, $G'$ has layered treewidth at most 3. That is, there is a tree decomposition $T'$ of $G'$, and a layering $V'_0,V_1',\dots$ of $G'$, such that each bag of $T'$ contains at most three vertices in each layer $V'_i$. For each vertex $v$ of $G'$, let $T'_v$ be the subtree of~$T'$ formed by the bags that contain $v$. 

Let $T$ be the decomposition of $G$ obtained by replacing each occurrence of a dummy vertex $x$ in a bag of $T'$ by  the tails of the two edges that cross at~$x$. We now show that $T$ is a tree-decomposition of $G$. For each vertex $v$ of $G$, let $T_v$ be the subgraph of $T$ formed by the bags that contain $v$. Let $G'_v$ be the subgraph of $G'$ induced by $v$ and the division vertices on the edges for which $v$ is the tail. Then $G'_v$ is connected. Thus $T'_v$, which is precisely the set of bags of $T'$ that intersect $G'_v$, form a (connected) subtree of $T'$. Moreover, for each oriented edge~$vw$ of~$G$, if $x$ is the division vertex of $vw$ adjacent to $w$, then $T'_x$ and $T'_w$ intersect. Since $T_v$ contains $T'_x$, and $T_w$ contains $T'_w$, we have that $T_v$ and $T_w$ intersect. Thus $T$ is a tree-decomposition of $G$. 

Note that $\dist_{G'}(v,w)\leq k+1$ for each edge $vw$ of $G$. Thus, if $v\in V'_i$ and $w\in V'_j$ then $|i-j|\leq k+1$. Let $V_0$ be the union of the first $k+1$ layers restricted to $V(G)$, let $V_1$ be the union of the second $k+1$ layers restricted to $V(G)$, and so on. That is, for $i\geq 0$, let $V_i:=V(G)\cap (V'_{(k+1)i}\cup V'_{(k+1)i+1}\cup \dots\cup V'_{(k+1)(i+1)-1})$. Then $V_0,V_1,\dots$ is a partition of $V(G)$. Moreover, if $v\in V_i$ and $w\in V_j$ for some edge $vw$ of $G$, then $|i-j|\leq 1$. Thus $V_1,V_2,\dots$ is a layering of $G$. 

Since each layer in $G$ consists of at most $k+1$ layers in $G'$, and each layer in $G'$ contains at most three vertices in a single bag, each of which are replaced by at most two vertices in $G$, the layered treewidth of this decomposition is at most $6(k+1)$. 
\end{proof}


\cref{Norine} and \cref{kPlanarLayeredTreewidth} imply the upper bound in \cref{kPlanarTreewidth}:

\begin{theorem}
\label{kPlanarTreewidthUpperBound}
Every $k$-planar $n$-vertex graph has treewidth at most $2\sqrt{6(k+1)n}$.
\end{theorem}

We now prove the corresponding lower bound. 

\begin{theorem}
\label{kPlanarTreewidthLowerBound}
For $1\leq k\leq \frac32 n$ there is a $k$-planar graph on $n$ vertices with treewidth at least $c\sqrt{kn}$ for some constant $c>0$.
\end{theorem}

\begin{proof}
Let $G$ be a cubic expander with $n$ vertices. Then $G$ has treewidth at least~$\epsilon n$ for some constant $\epsilon>0$ (see for example \citep{GM-JCTB}). Consider a straight-line drawing of $G$. Clearly, each edge is crossed less than $|E(G)|=\frac32 n$ times.  Subdivide each edge of $G$ at most $\frac{3n}{2k}$ times to produce a $k$-planar graph~$G'$ with $n'$~vertices, where $n'\leq n +\frac{3n}{2}\,\frac{3n}{2k}<\frac{4n^2}{k}$. 
Subdivision does not change the treewidth of a graph. Thus $G'$ has treewidth at least $\epsilon n\geq \frac{\epsilon}{2}\sqrt{kn'}$.
\end{proof}

Combining the bound of \cref{kPlanarTreewidthUpperBound} with the trivial upper bound $\tw(G)\leq n$ for $k\ge n$ shows that the maximum treewidth of $k$-planar $n$-vertex graphs is $\Theta(\min\{n,\sqrt{kn}\})$ for arbitrary $k$ and $n$. This completes the proof of \cref{kPlanarTreewidth}.

\section{\boldmath $(g,k)$-Planar Graphs} 
\label{gkPlanarGraphs}

Recall that a graph is $(g,k)$-planar if it can be drawn in a surface of Euler genus at most $g$ with at most $k$ crossings on each edge. The proof method used in \cref{kPlanarLayeredTreewidth} in conjunction with \cref{DMW} leads to the following theorem.

\begin{theorem}
\label{gkPlanarLayeredTreewidth}
Every $(g,k)$-planar graph $G$ has layered treewidth at most $(4g+6)(k+1)$.
\end{theorem}

\begin{proof}
Consider a drawing of $G$ with at most $k$ crossings per edge on a surface $\Sigma$ of Euler genus $g$. Arbitrarily orient each edge of $G$. Let $G'$ be the graph obtained from $G$ by replacing each crossing by a new degree-4 vertex. Then $G'$ is embedded in $\Sigma$ with no crossings, and thus has Euler genus at most $g$. By \cref{DMW}, $G'$ has layered treewidth at most $2g+3$. That is, there is a tree decomposition $T'$ of $G'$, and a layering $V'_0,V_1',\dots$ of $G'$, such that each bag of $T'$ contains at most $2g+3$ vertices in each layer $V'_i$. For each vertex $v$ of $G'$, let $T'_v$ be the subtree of $T'$ formed by the bags that contain $v$. 

Let $T$ be the decomposition of $G$ obtained by replacing each occurrence of a dummy vertex $x$ in a bag of $T'$ by  the tails of the two edges that cross at~$x$. We now show that $T$ is a tree-decomposition of $G$. For each vertex $v$ of $G$, let $T_v$ be the subgraph of $T$ formed by the bags that contain $v$. Let $G'_v$ be the subgraph of $G'$ induced by $v$ and the division vertices on the edges for which $v$ is the tail. Then $G'_v$ is connected. Thus $T'_v$, which is precisely the set of bags of $T'$ that intersect $G'_v$, form a (connected) subtree of $T'$. Moreover, for each oriented edge $vw$ of $G$, if $x$ is the division vertex of $vw$ adjacent to $w$, then $T'_x$ and $T'_w$ intersect. Since $T_v$ contains $T'_x$, and $T_w$ contains $T'_w$, we have that $T_v$ and $T_w$ intersect. Thus $T$ is a tree-decomposition of $G$. 

Note that $\dist_{G'}(v,w)\leq k+1$ for each edge $vw$ of $G$. Thus, if $v\in V'_i$ and $w\in V'_j$ then $|i-j|\leq k+1$. Let $V_0$ be the union of the first $k+1$ layers restricted to $V(G)$, let $V_1$ be the union of the second $k+1$ layers restricted to $V(G)$, and so on. That is, for $i\geq 0$, let $V_i:=V(G)\cap (V'_{(k+1)i}\cup V'_{(k+1)i+1}\cup \dots\cup V'_{(k+1)(i+1)-1})$. Then $V_0,V_1,\dots$ is a partition of $V(G)$. Moreover, if $v\in V_i$ and $w\in V_j$ for some edge $vw$ of $G$, then $|i-j|\leq 1$. Thus $V_1,V_2,\dots$ is a layering of $G$. Since each layer in $G$ consists of at most $k+1$ layers in $G'$, and each layer in $G'$ contains at most $2g+3$ vertices in a single bag, each of which is replaced by at most two vertices in $G$, the layered treewidth of this decomposition is at most $(4g+6)(k+1)$. 
\end{proof}

\cref{gkPlanarLayeredTreewidth} and \cref{Norine} imply:

\begin{theorem}
\label{gkPlanarTreewidthUpperBound}
Every $n$-vertex $(g,k)$-planar graph has treewidth at most $$2\sqrt{(4g+6)(k+1)n}.$$
\end{theorem}

We now show that the bounds in \cref{gkPlanarLayeredTreewidth} and \cref{gkPlanarTreewidthUpperBound} are tight up to a constant factor. 

\begin{theorem}
\label{gkPlanarTreewidthLowerBound}
For all $g,k\geq 0$ and infinitely many $n$ there is an $n$-vertex $(g,k)$-planar graph with treewidth 
$\Omega(\sqrt{(g+1)(k+1)n})$ and layered treewidth $\Omega((g+1)(k+1))$. 
\end{theorem}

The proof of this result depends on the separation properties of the $p\times q\times r$ grid graph (which is $(r-1)$-planar). 
The next two results are not optimal, but have simple proofs and are all that is needed for the main proof that follows. 

\begin{lemma}
\label{TwoDimGrid}
For $q\geq (\frac{1}{1-\epsilon})r$, every $\epsilon$-separator of the $q \times r$ grid graph has size at least $r$.
\end{lemma}

\begin{proof}
Let $S$ be a set of at most $r-1$ vertices in the $q\times r$ grid graph. Some row $R$ avoids $S$, and at least $q-r+1$ columns avoid $S$. The union of these columns with $R$ induces a connected subgraph with at least $(q-r+1)r>\epsilon qr$ vertices. Thus $S$ is not an $\epsilon$-separator. 
\end{proof}

\begin{lemma}
\label{ThreeDimGrid}
For  $p\geq q\geq (\frac{1}{1-\epsilon})r$, every $\epsilon$-separator of the $p\times q \times r$ grid graph has size at least $(\frac{1-\epsilon}{1+\epsilon})qr$.
\end{lemma}

\begin{proof}
Let $G$ be the $p\times q \times r$ grid graph. Let $n:=|V(G)|=pqr$. 
Let $S$ be an $\epsilon$-separator of $G$. 
Let $A_1,\dots,A_c$ be the components of $G-S$. 
Thus $|A_i|\leq\epsilon n$. 
For $x\in[p]$, let $G_x :=\{(x,y,z):y\in[q],z\in[r]\}$ called a \emph{slice}. 
Say $G_x$ \emph{belongs} to $A_i$ and $A_i$ \emph{owns} $G_x$ if $|A_i\cap G_x| \geq \frac{1+\epsilon}{2} qr$.
Clearly, no two components own the same slice. 
First suppose that at least two components each own a slice. That is, $G_v$ belongs to $A_i$ and $G_w$ belongs to $A_j$ for some $v<w$ and $i\neq j$. 
Let $X:=\{(y,z):(v,y,z)\in G_v,\,(w,y,z)\in G_w\}$. Then $|X|\geq 2(\frac{1+\epsilon}{2})qr-qr=\epsilon qr$. 
For each $(y,z)\in X$, the `straight' path $(v,y,z), (v+1,y,z),\dots,(w,y,z)$ contains some vertex in $S$. 
Since these paths are pairwise disjoint, $|S|\geq|X|\geq\epsilon qr\geq \frac{1-\epsilon}{1+\epsilon}qr$ (since $\epsilon>\frac12$). 
Now assume that at most one component, say $A_1$, owns a slice. 
Say $A_1$ owns $t$ slices. Thus $t(\frac{1+\epsilon}{2})qr \leq |A_i|\leq \epsilon pqr$ and $t\leq \frac{2\epsilon}{1+\epsilon}p$. 
Hence, at least $(1-\frac{2\epsilon}{1+\epsilon})p$ slices belong to no component. 
For such a slice $G_v$, each component of $G_v-S$ is contained in some $A_i$ and thus has at most $(\frac{1+\epsilon}{2})qr$ vertices. 
That is, $S\cap G_v$ is a $(\frac{1+\epsilon}{2})$-separator of the $q\times r$ grid graph induced by $G_v$. 
 By \cref{TwoDimGrid}, $|S\cap G_v|\geq r$. Thus $|S|\geq (1-\frac{2\epsilon}{1+\epsilon})pr\geq (\frac{1-\epsilon}{1+\epsilon})qr$.
\end{proof}

Note that \cref{RS} and \cref{ThreeDimGrid} imply:

\begin{corollary}
\label{TreewidthThreeDimGrid}
For  $p\geq q\geq 2r$, the $p\times q \times r$ grid graph has treewidth at least $\frac{1}{3}qr$.
\end{corollary}

This lower bound is within a constant factor of optimal, since \citet{OS11} proved that the $p\times q \times r$ grid graph has pathwidth, and thus treewidth, at most $qr$.

\begin{proof}[Proof of \cref{gkPlanarTreewidthLowerBound}] 
Let $r:= k+1$. 

First suppose that $g\leq 19$. Let $G$ be the $q\times q\times r$ grid graph where $q\geq 2r$. As observed above, $G$ is $k$-planar and thus $(g,k)$-planar. \cref{ThreeDimGrid} implies that every $\frac12$-separator of $G$ has size at least $\frac13 qr$. \cref{RS} thus implies that $G$ has treewidth at least $\frac13 qr-1$, which is $\Omega(\sqrt{(g+1)(k+1)n})$, as desired. 

Now assume that $g\geq 20$. By \cref{Expander} there is a 4-regular expander $H$ on $m:=\floor{\frac{g}{4}}\geq 5$ vertices. Thus $H$ has $2m$ edges, $H$ embeds in the orientable surface with $2m$ handles, and thus has Euler genus at most $4m\leq g$. We may assume that $q:=\sqrt{n/rm}$ is an integer with $q\geq 8r$. Let $G$ be obtained from $H$ by replacing each vertex $v$ of $H$ by a copy of the $q\times q\times r$ grid graph with vertex set $D_v$, and replacing each edge $vw$ of $H$ by a matching of $qr$ edges, so that $G[D_v\cup D_w]$ is a $2q\times q\times r$ grid, as shown in \cref{BlowUpGrid}. Thus $G$ is $(g,k)$-planar with $q^2rm=n$ vertices. 

\begin{figure}[t]
\centering
\includegraphics[width=0.65\textwidth]{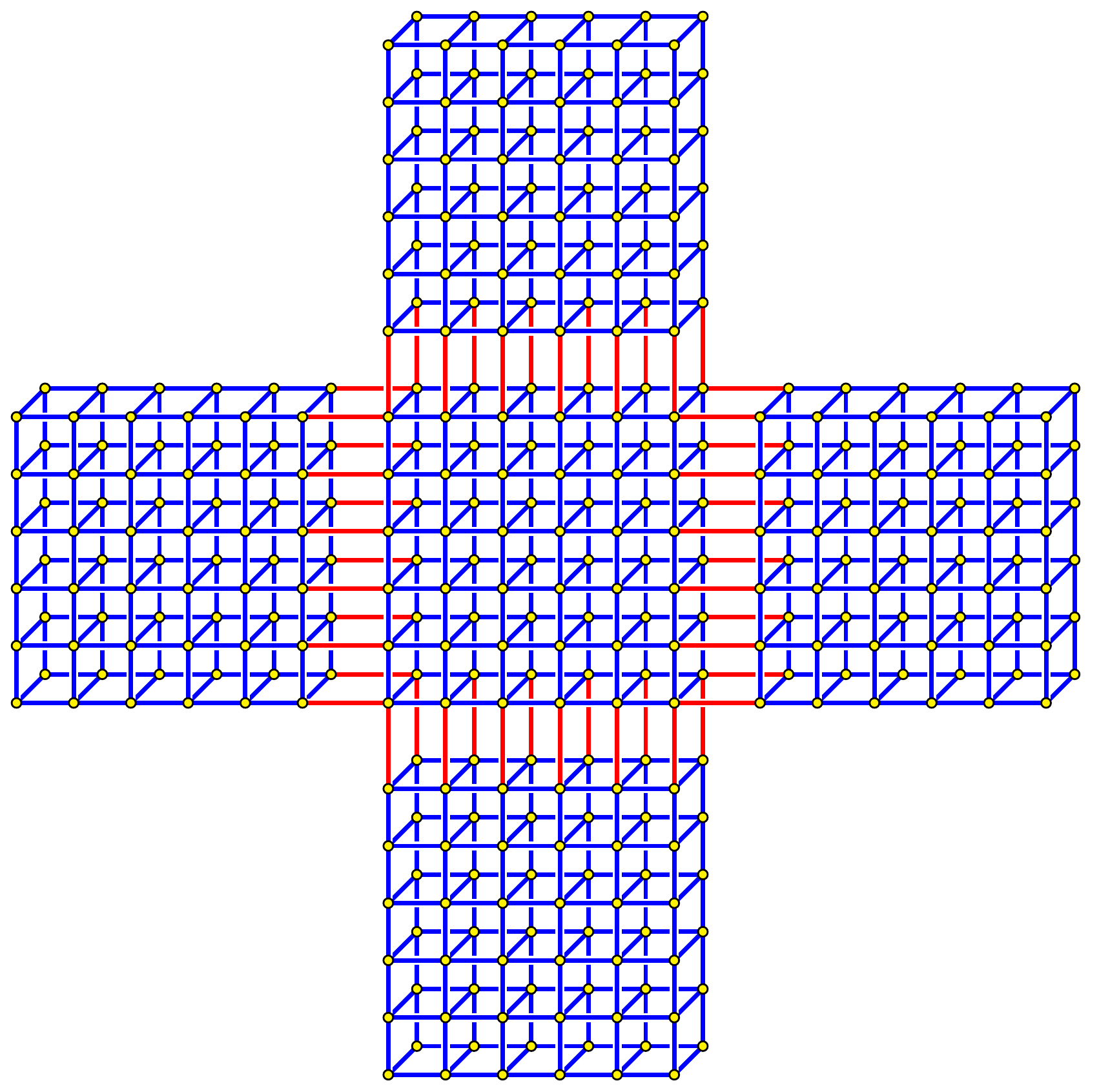}
\caption{Construction of $G$ in the proof of \cref{gkPlanarTreewidthLowerBound}.}
\label{BlowUpGrid}
\end{figure}

Let $S$ be a $\frac12$-separator in $G$. Let $A_1,\dots,A_c$ be the components of $G-S$. Thus $|A_i|\leq\frac{1}{2}n$ for $i\in[c]$. Initialise sets $S':=A'_1:=\dots:=A'_c:=\emptyset$. 

For each vertex $v$ of $H$, if $|S\cap D_v|\geq\frac{qr}{14}$ then put $v\in S'$. 
Otherwise, $|S\cap D_v| < \frac{qr}{14}$. 
Note that \cref{ThreeDimGrid} is applicable with $\epsilon=\frac{13}{15}$ since $q\geq 8r>\frac{1}{1-13/15}r$ and $\frac{1-13/15}{1+13/15}=\frac{1}{14}$. \cref{ThreeDimGrid}  thus implies that  $S\cap D_v$ is not a $\frac{13}{15}$-separator. Hence some component of $D_v-S$ has at least $\frac{13}{15}q^2r$ vertices. 
Since $\frac{13}{15}>\frac{1}{2}$, exactly one component of $D_v-S$ has at least $\frac{13}{15}q^2r$ vertices. 
This component is a subgraph of $A_i$ for some $i\in[c]$; add $v$ to $A'_i$. Thus $S',A'_1.\dots,A'_c$ is a partition of $V(H)$. 

We now prove that $S'$ is a $\frac{15}{26}$-separator in $H$. Suppose that $v\in A'_i$ and $w\in A'_j$ for some edge $vw$ of $H$. Let $D$ be the vertex set of the $2q\times q\times r$ grid graph induced by $D_v\cup D_w$.  Since $v\not\in S'$ and $w\not\in S'$, we have $|S\cap D_v|<\frac{qr}{14}$ and $|S\cap D_w|<\frac{qr}{14}$. Thus $|S\cap D| <\frac{qr}{7}$. Note that \cref{ThreeDimGrid} is applicable with $\epsilon=\frac{3}{4}$ since $q\geq 8r>\frac{1}{1-3/4}r$ and $\frac{1-3/4}{1+3/4}=\frac{1}{7}$. \cref{ThreeDimGrid}  thus implies that  $S\cap D$ is not a $\frac{3}{4}$-separator of $G[D]$. 
Hence some component $X$ of $G[D]-S$ contains at least $\frac{3}{4}|D|=\frac{3}{2}q^2r$ vertices. Each of $D_v$ and $D_w$ can contain at most $q^2r$ vertices in $X$. Thus $D_v$ and $D_w$ each contain at least $\frac{1}{2}q^2r$ vertices in $X$. Thus, by construction, $v$ and $w$ are in the same $A'_i$. That is, there is no edge of $H$ between distinct $A'_i$ and $A'_j$, and each component of $H-S'$ is contained in some $A'_i$. For each $i\in[c]$, we have $\frac{1}{2}q^2rm\geq|A_i|\geq \frac{13}{15}q^2r|A'_i|$ implying $|A'_i| \leq \frac{15}{26}m$. Therefore $S'$ is a $\frac{15}{26}$-separator in $H$. 

By \cref{Expander}, $|S'|\geq\beta m$ for some constant $\beta>0$. Thus $|S|\geq \frac{qr}{14}|S'|\geq\frac{\beta}{14}mqr$. By \cref{RS}, $G$ has treewidth at least $\frac{\beta}{14}mqr-1=\frac{\beta}{14}\sqrt{mrn}-1\geq \Omega(\sqrt{g(k+1)n})$, as desired. 

Finally, by \cref{Norine}, if $G$ has layered treewidth $\ell$ then $\Omega(\sqrt{g(k+1)n}) \leq \tw(G) \leq 2\sqrt{\ell n}$, implying $\ell \geq \Omega((g+1)(k+1))$. 
\end{proof}

Note that the proof of \cref{gkPlanarTreewidthLowerBound} in the case $k=0$ is very similar to that of \citet{GHT-JAlg84}. 

For $gk\geq n$ the trivial upper bound of $\tw(G)\leq n$ is better than that given in \cref{gkPlanarTreewidthUpperBound}. We conclude that the maximum treewidth of $(g,k)$-planar $n$-vertex graphs is $\Theta(\min\{n,\sqrt{(g+1)(k+1)n}\})$ for arbitrary $g,k,n$. This completes the proof of \cref{gkPlanarTreewidth}.

\section{Drawings with Few Crossings per Edge}

This section studies the following natural conjecture: for every  surface $\Sigma$ of Euler genus $g$, every graph $G$ with $m$ edges has a drawing in $\Sigma$ with $O(\frac{m}{g+1})$ crossings per edge. This conjecture is trivial at both extremes: with $g=0$, every  graph has a straight-line drawing in the plane (and therefore a drawing in the sphere) with at most $m$ crossings per edge, and with $g=2m$, every graph has a crossing-free drawing in the orientable surface with one handle per edge. Moreover, if this conjecture is true, it would provide a simple proof of \cref{gkPlanarTreewidthLowerBound} in the same manner as the proof of \cref{kPlanarTreewidthLowerBound}. 

Our starting point is the following well-known result of \citet[Theorem~22, p.~822]{LR99}:

\begin{theorem}[\citep{LR99}]
\label{thm:leighton-rao}
Let $G$ be a graph with bounded degree and $n$ vertices, mapped one-to-one onto the vertices of an expander graph $H$. Then the edges of $G$ can be mapped onto paths in $H$ so that each path has length $O(\log n)$ and each edge of $H$ is used by $O(\log n)$ paths.
\end{theorem}

It is straightforward to extend this result to regular graphs $G$ of unbounded degree, with the number of paths per edge of $H$ increasing in proportion to the degree. However, there are two difficulties with using it in our application. First, it does not directly handle graphs in which there is considerable variation in degree from vertex to vertex: in such cases we would want the number of paths per edge to be controlled by the average degree in~$G$, but instead it is controlled by the maximum degree. And second, it does not allow us to control separately the sizes of $G$ and $H$; instead, both must have the same number of vertices. To handle these issues, we do  not map the vertices of our input graph $G$ directly to the vertices of an expander $H$; instead, we keep the vertices of $G$ and the vertices of $H$ disjoint from each other, connecting them by a bipartite graph that balances the degrees, according to the following lemma.

\begin{lemma}
\label{lem:load-balance}
Let $d_1,d_2,\dots,d_n$ be a sequence of positive integers, and let $q$ be a positive integer. Then there exists a bipartite graph with colour classes $\{v_1,\dots,v_n\}$ and $\{w_1,\dots,w_q\}$, at most $n+q-1$ edges, and a labelling of the edges with positive integers, such that
\begin{itemize}
\item each vertex $v_i$ is incident to a set of edges whose labels sum to $d_i$, and
\item each pair of distinct vertices $w_i$ and $w_j$ are incident to sets of edges whose label sums differ by at most $1$.
\end{itemize}
\end{lemma}

\begin{proof}
Preassign label sums of $\lfloor \sum d_i/q \rfloor$ or $\lceil \sum d_i/q \rceil$ to each vertex $w_i$ so that the resulting values sum to $\sum d_i$. We will construct a bipartite graph and a labelling whose sums match the numbers $d_1,\dots,d_n$ on one side of the bipartition and whose sums match the preassigned numbers on the other side.

Build this graph and its labelling one edge at a time, starting from a graph with no edges. At each step, let $v_i$ and $w_j$ be the vertices on each side of the bipartition with the smallest indices whose edge labels do not yet sum to the required values, add an edge from $v_i$ to $w_j$, and label this edge with the largest integer that does not exceed the required sum on either vertex.

Each step completes the sum for at least one vertex. Because the required values on the two sides of the bipartition both sum to $\sum d_i$, the final step completes the sum for two vertices, $v_n$ and $w_q$. Therefore, the total number of steps, and the total number of edges added to the graph, is at most $n+q-1$.
\end{proof}

\begin{figure}[t]
\centering
\centering\includegraphics[width=0.9\textwidth]{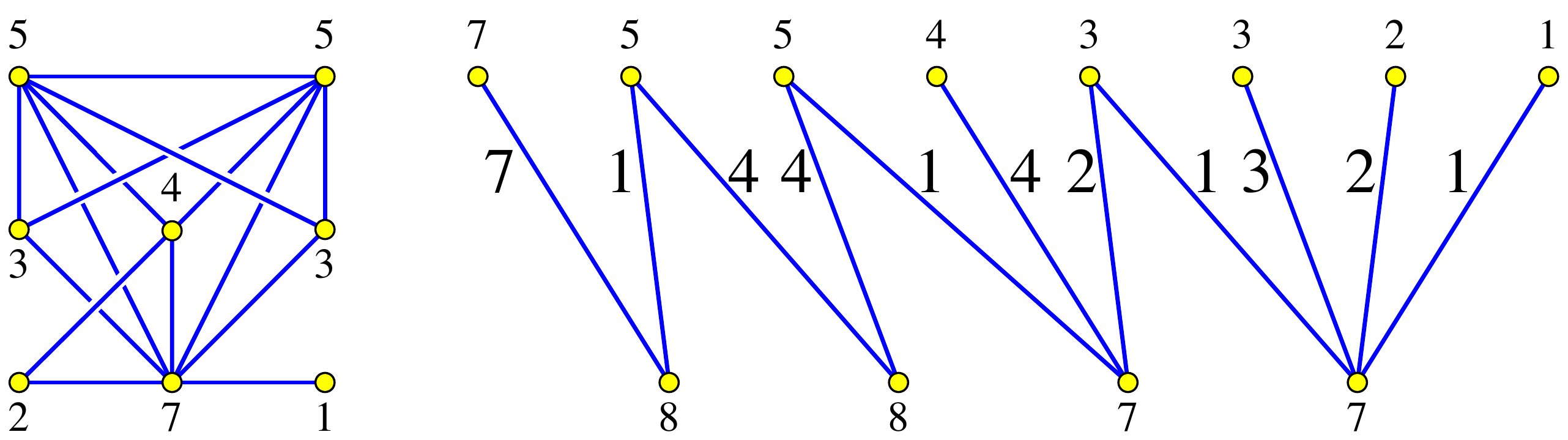}
\caption{A graph (left) with degree sequence $7,5,5,4,3,3,2,1$ and a bipartite graph (right) formed from this degree sequence by Lemma~\ref{lem:load-balance}. The large numbers are the edge labels of the lemma, and the small numbers along the top and bottom of the bipartite graph give the sums of incident edge labels at each vertex. The top sums match the given degree sequence, while the bottom sums all differ by at most 1.}
\label{fig:degseq-loadbal}
\end{figure}

By combining this load-balancing step with the Leighton-Rao expander-routing scheme, we may obtain a more versatile mapping of our given graph $G$ to a host graph $H$, with better control over the genus of the surface we obtain from $H$. This genus will be determined by the \emph{cyclomatic number} of $H$, where the cyclomatic number of a graph with $n$ vertices and $m$ edges is $m-n+1$. This number is the dimension of the cycle space of the graph, and the first Betti number of the topological space obtained from the graph by replacing each edge by a line segment.

\begin{lemma}
\label{lem:hosted}
Let $G$ be an arbitrary graph, with $m$ edges, and let $Q$ be a $q$-vertex bounded-degree expander graph.
Then there exists a host graph $H$, a one-to-one mapping of the vertices of $G$ to a subset of vertices of $H$, and a mapping of the edges of $G$ to paths in $H$, with the following properties:
\begin{itemize}
\item The vertices of $H$ that are not images of vertices in $G$ induce a subgraph isomorphic to~$Q$.
\item The image of an edge $e$ in $G$ forms a path of length $O(\log q)$ that starts and ends at the image of the endpoints of $e$, and  passes through the image of no other vertex of $G$.
\item Each vertex of $H$ that is not an image of a vertex in $G$ is crossed by $O((m\log q)/q)$ paths.
\item The cyclomatic number of $H$ is $O(q)$.
\end{itemize}
\end{lemma}

\begin{proof}
Let the vertices of $G$ be $u_1,\dots,u_n$. 
Apply \cref{lem:load-balance} to the degree sequence of $G$ to form a bipartite graph~$H$ with bipartition $\{v_1,\dots,v_n\},\{w_1,\dots,w_q\}$. Then add edges between pairs of vertices $(w_i,w_j)$ so that $\{w_1,\dots,w_q\}$ induces a subgraph isomorphic to $Q$. In this way, each vertex $u_i$ in $G$ is mapped to a vertex $v_i$ in $H$ so that the mapping is one-to-one and the unmapped vertices form a copy of $Q$, as required. The cyclomatic number of $H$ equals the cyclomatic number of $Q$, plus $n+q-1$ (for the added edges in the bipartite graph), minus~$n$ (for the added vertices relative to $Q$). These two added and subtracted terms cancel, leaving the cyclomatic number of $Q$ plus $q-1$, which is $O(q)$ as required.

It remains to find paths in $H$ corresponding to the edges in $G$. Assign each edge $u_iu_j$ of $G$ to a pair of vertices $(w_{i'},w_{j'})$ adjacent to the images $v_i$ and $v_j$ in $H$, so that the number of edges of $G$ assigned to each edge between $\{v_1,\dots,v_n\}$ and $\{w_1,\dots,w_q\}$ equals the corresponding label. Complete each path by applying \cref{thm:leighton-rao} to the copy of $Q$; this gives paths of length $O(\log q)$ connecting each pair $(w_{i'},w_{j'})$ obtained in this way. These pairs do not form a bounded-degree graph, but they can be partitioned into $O(m/q)$ bounded-degree graphs, each of which causes each vertex in the copy of $Q$ to be crossed $O(\log q)$ times. Combining these suproblems, each vertex in the copy of $Q$ is crossed by a total of $O((m\log q)/q)$ paths, as required.
\end{proof}

We are now ready to prove the existence of embeddings with small local crossing number, on surfaces of arbitrary genus.

\begin{proof}[Proof of \cref{thm:few-crossings}] 
Given a graph $G$, to be embedded on a surface with at most $g$ handles and with few crossings per edge, choose $q$ so that the $O(q)$ bound on the cyclomatic number of the graph $H$ in \cref{lem:hosted} is at most $g$, and apply \cref{lem:hosted} to find a graph $H$ and a mapping from $G$ to $H$ obeying the conditions of the lemma.

To turn this mapping into the desired embedding of $G$, replace each vertex of degree $d$ in $H$ by a sphere, punctured by the removal of $d$ unit-radius disks, and form a surface (as a cell complex, not necessarily embedded into three-dimensional space) by replacing each edge $xy$ of $H$ by a unit-radius cylinder connecting boundaries of removed disks on the spheres for vertices $x$ and $y$. The number of handles on the resulting surface (shown in \cref{fig:ball-and-stick}) equals the cyclomatic number of~$H$, which is at most $g$. 

\begin{figure}[h]
\centering
\includegraphics[width=2.25in]{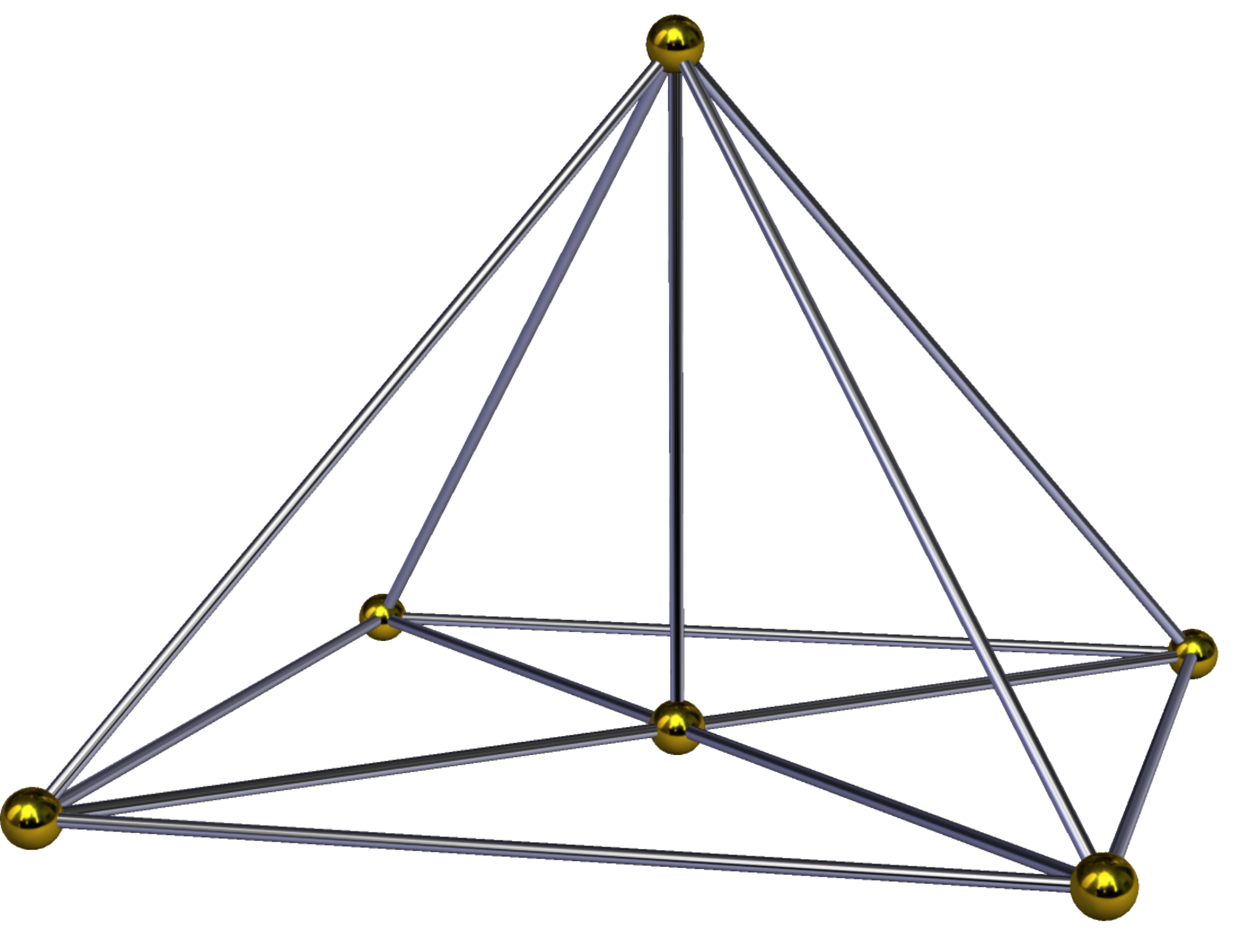}
\caption{A topological surface obtained by replacing each vertex of a graph by a punctured sphere, and each edge of the graph by a cylinder connecting two punctures. Image \href{https://commons.wikimedia.org/wiki/File:Square_pyramid_pyramid.png}{Square\_pyramid\_pyramid.png} by Tom Ruen on Wikimedia commons, made available under a Creative Commons CC-BY-SA 4.0 International license.}
\label{fig:ball-and-stick}
\end{figure}

Embed each vertex of $G$ as an arbitrarily chosen point on the sphere of the corresponding vertex of~$H$, and each edge of $G$ as a curve through the sequence of spheres and cylinders corresponding to its path in~$H$. Choose this embedding so that no intersection of edge curves occurs within any of the cylinders, and so that every pair of edges that are mapped to curves on the same sphere meet at most once, either at a crossing point or a shared endpoint.

Because the spheres that contain vertices of $G$ only contain curves incident to those vertices, they do not have any crossings. Each edge is mapped to a curve through $O(\log g)$ of the remaining spheres, and can cross at most $O((m\log g)/g)$ other curves within each such sphere. Therefore, the maximum number of crossings per edge is $O((m\log^2 g)/g)$.
\end{proof}

\section{Map Graphs}

The following characterisation of map graphs makes them easier to deal with (and is well known in the $g=0$ case \citep{CGP02}). Consider a bipartite graph $H$ with bipartition $\{A,B\}$. Define the \emph{half-square} graph $H^2[A]$ with vertex set $A$, where two vertices in $A$ are adjacent if and only if they have a common neighbour in $B$. 

\begin{lemma}
\label{MapGraphCharacterisation}
A graph $G$ is a $(g,d)$-map graph if and only if $G$ is isomorphic to $H^2[A]$ for some bipartite graph $H$ with Euler genus at most $g$ and bipartition $\{A,B\}$, where vertices in $B$ have maximum degree at most $d$. 
\end{lemma}

\begin{proof}
$(\Longrightarrow$) Say $G$ is a $(g,d)$-map graph defined with respect to some graph $G_0$ embedded in a surface of Euler genus $g$, where each face of $G_0$ is a nation or a lake. Let $H$ be the bipartite graph with bipartition $\{A,B\}$, where  $A$ is the set of nations of $G_0$
and  $B:=V(G_0)$, where a vertex $v\in A$ is adjacent to a vertex $w\in B$ if $w$ is incident to the face in $G_0$ corresponding to $v$. Then $H$ embeds in the same surface as $G_0$, and by definition, $G$ is isomorphic to $H^2[A]$. The degree of a vertex $w$ in $B$ equals the number of nations incident to $w$ in $G_0$, which is at most $d$. 

$(\Longleftarrow$) Consider a bipartite graph $H$ with bipartition $\{A,B\}$, where vertices in $B$ have maximum degree at most $d$. From an embedding of $H$ in a surface of Euler genus $g$, construct an embedded graph $G_0$ with vertex set $V(G_0):=B$, where $uw\in E(G_0)$ whenever $vu$ and $vw$ are consecutive edges incident to some vertex $v\in A$ in the embedding of $H$. So, for each vertex $v\in A$, if $vw_1,vw_2,\dots,vw_p$ is the cyclic order of edges incident to $v$ in the embedding of $H$, then $(w_1,w_2,\dots,w_p)$ is a face of $G_0$, which we label as a nation. Label every other face of $G_0$ as a lake. Note that a lake occurs whenever, for some $k\geq 3$, there is a face $(v_1,w_1,v_2,w_2,\dots,v_k,w_k)$ of $H$ with $v_i\in A$ and $w_i\in B$. Then $(w_1,w_2,\dots,w_k)$ is a lake of $G_0$. By construction, 
the nations of $G_0$ are in 1--1 correspondence with vertices in $A$, and for each vertex $w$ of $G_0$, the number of nations incident to $w$ equals the degree of $w$ in $H$, which is at most $d$. Two nations are incident to a common vertex $w$ of $G_0$ if and only if the corresponding vertices in $A$ are both adjacent to $w$ in $H$. Thus $H^2[A]$ is isomorphic to the $(g,d)$-map graph associated with $G_0$. 
\end{proof}

\begin{lemma} 
\label{MapGraph}
Let $H$ be a bipartite graph with bipartition $\{A,B\}$ and layered treewidth $k$ with respect to some layering $A_1,B_1,A_2,B_2,\dots,A_t,B_t$, where $A=A_1\cup\dots\cup A_t$ and $B=B_1\cup \dots\cup B_t$. Then the half-square graph $G = H^2[A]$ has layered treewidth at most $k(2d+1)$ with respect to layering $A_1,A_2,\dots,A_t$, where $d$ is the maximum degree of vertices in $B$. 
\end{lemma}

\begin{proof} 
Let $T$ be the given tree decomposition of $H$. For each bag $X$ and for each vertex $w$ in $B \cap X$, replace $w$ in $X$ by $N_H(w)$ and delete $w$ from $X$. Each vertex $v$ in $A$ is now precisely in the bags that previously intersected $N_H(v) \cup \{v\}$. Since $N_H(v) \cup\{v\}$ induces a connected subgraph of $H$, the bags that now contain $v$ form a connected subtree of $T$. 

Consider an edge $uv \in E(G)$. Then $u,v \in N_H(w)$ for some $w \in B$. By construction, $u$ and $v$ are in a common bag, and we have a tree decomposition of $G$. Say $u\in A_i$ and $v\in A_j$ and $w\in B_\ell$. Since $uw\in E(H)$ and $A_1,B_1,A_2,B_2,\dots,A_t,B_t$ is a layering of $H$, we have $\ell\in\{i,i-1\}$. Similarly, $\ell\in\{j,j-1\}$. Thus $|i-j|\leq 1$. Hence $A_1,A_2,\dots,A_t$ is a layering of $G$. 

We now upper bound $|X \cap A_i|$ for each bag $X$ and layer $A_i$. If $v \in X \cap
A_i$, then (1) $v$ was in $X$ in the given tree decomposition of $H$, or (2)
$v$ is adjacent to some vertex $w$ in $( B_i \cup B_{i+1} ) \cap X$. 
Thus, the number of such vertices $w$ is at most $2k$. Each such vertex $w$
contributes at most $d$ vertices to $X$. The number of type-(1) vertices $v$
is at most $k$. Thus $|X \cap A_i|\leq k + 2kd$. 
\end{proof}

The next lemma is a minor technical strengthening of \cref{DMW}. We sketch the proof for completeness. 

\begin{lemma}
\label{StrongDMW}
Let $V_1,V_2,\dots,V_t$ be a bfs layering of a connected graph $G$ of Euler genus at most $g$. 
Then $G$ has a tree decomposition of layered width $2g+3$ with respect to $V_1,V_2,\dots,V_t$.
\end{lemma}

\begin{proof}[Proof Sketch]
Let $r$ be the vertex for which $V_i=\{v\in V(G):\dist(v,r)=i\}$. 
Let $T$ be a bfs tree of $G$ rooted at $r$. 
For each vertex $v$ of $G$, let $P_v$ be the vertex set of the $vr$-path in $T$. 
Thus if $v\in V_i$, then  $P_v$ contains exactly one vertex in  $V_j$ for $j\in\{0,\dots,i\}$.

Let $G'$ be a triangulation of $G$ with $V(G')=V(G)$ (allowing parallel edges on distinct faces). 
Let $F$ be the set of faces of $G'$. Say $G$ has  $n$ vertices.  
By Euler's formula, $|F|=2n+2g-4$ and $|E(G')|=3n+3g-6$. 

Let $D$ be the subgraph of the dual of $G'$ with vertex set $F$, where two vertices are adjacent if the corresponding faces share an edge not in $T$. 
Thus $|V(D)|=|F|=2n+2g-4$ and $|E(D)|=|E(G')|-|E(T)|=(3n+3g-6)-(n-1)=2n+3g-5$. 
\citet{DMW13} proved that $D$ is connected. 

Let $T^*$ be a spanning tree of $D$. Thus $|E(T^*)|=|V(D)|-1=2n+2g-5$. Let $X:=E(D)-E(T^*)$. 
Thus $|X|=(2n+3g-5)-(2n+2g-5)=g$. For each face $f=xyz$ of $G'$, let 
$C_f:= \cup \{P_a\cup P_b:ab\in X\} \cup P_x\cup P_y\cup P_z$. 
Since $|X|=g$ and each $P_v$ contains at most one vertex in each layer,  $C_f$ contains at most $2g+3$ vertices in each layer.
\citet{DMW13} proved that  $(C_f:f\in F)$ is a $T^*$-decomposition of $G$. 
\end{proof}

We now present the main results of this section. 

\begin{theorem} 
\label{MapGraphLayeredTreewidth}
Every $(g,d)$-map graph has layered treewidth at most $(2g+3)(2d+1)$.
\end{theorem}


\begin{proof}
Let $G$ be a $(g,d)$-map graph. Since the layered treewidth of $G$ equals the maximum layered treewidth of the components of $G$, we may assume that $G$ is connected. By \cref{MapGraphCharacterisation}, $G$ is isomorphic to $H^2[A]$ for some bipartite graph $H$ with bipartition $\{A,B\}$ and Euler genus $g$, where vertices in $B$ have degree at most $d$ in $H$. Since $G$ is connected, $H$ is connected. Fix a vertex $r\in A$. For $i\geq 1$, let $A_i$ be the set of vertices of $H$ at distance $2i-2$ from $r$, and let $B_i$ be the set of vertices of $H$ at distance $2i-1$ from $r$. Since $H$ is bipartite and connected, $A=A_1\cup\dots,A_t$ and $B=B_1\cup\dots\cup B_t$ for some $t$, and $A_1,B_1,\dots,A_t,B_t$ is a bfs layering of $H$. By \cref{StrongDMW}, $H$ has a tree decomposition of layered width $2g+3$ with respect to $A_1,B_1,\dots,A_t,B_t$. By \cref{MapGraph}, $H^2[A]$ and thus $G$ has layered treewidth at most $(2g+3)(2d+1)$.
\end{proof}

\cref{Norine} and \cref{MapGraphLayeredTreewidth} imply:

\begin{theorem} 
\label{MapGraphTreewidth}
Every  $n$-vertex $(g,d)$-map graph has treewidth at most $$2\sqrt{(2g+3)(2d+1)n}-1.$$
\end{theorem}

Note that \citet{Chen01} proved that $d$-map graphs have separators of size $O(\sqrt{dn})$, which is implied by 
\cref{MapGraphTreewidth} and \cref{RS}.

%

We now show that \cref{MapGraphTreewidth} and thus \cref{MapGraphLayeredTreewidth} are tight. For integers $p,q,r \geq 1$, let $Y_{p,q,r}$ be the plane graph obtained from the $(p+1)\times(q+1)$ grid graph by subdividing each edge $r-1$ times, and then adding a vertex adjacent to the $4r$ vertices of each internal face. As illustrated in \cref{YpqrZpqr},  $Y_{p,q,r}$ is an internal triangulation with maximum degree $d:=4r$. Label each internal face of $Y_{p,q,r}$ as a nation, label the external face as a lake, and let $Z_{p,q,r}$ be the associated $d$-map graph. 

\begin{figure}[h]
\centering
\includegraphics[scale=0.5]{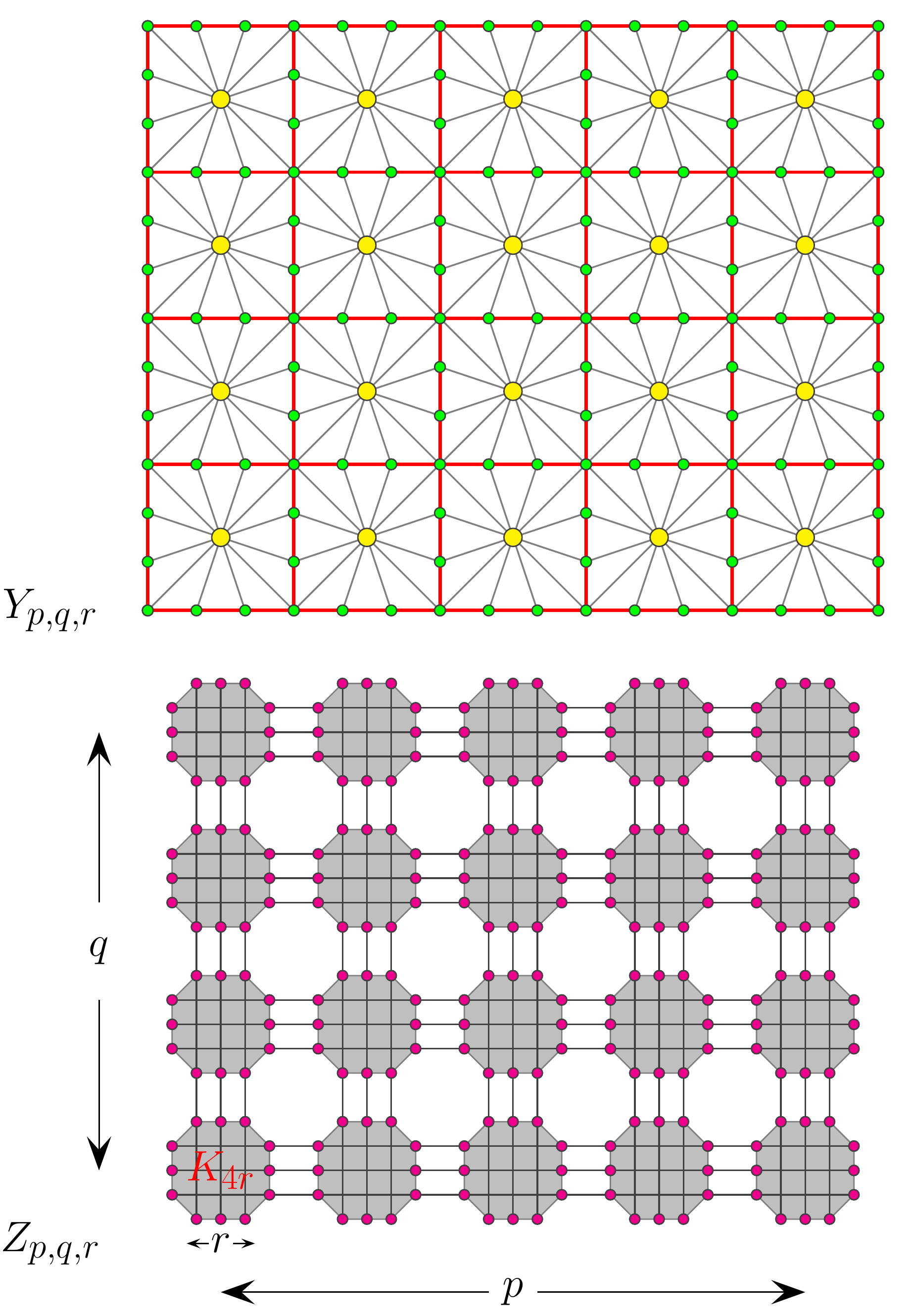}
\caption{$Z_{p,q,r}$ is the map graph of $Y_{p,q,r}$. The bottom figure shows the rows and columns of $Z_{p,q,r}$ (and omits other edges). }
\label{YpqrZpqr}
\end{figure}

\begin{lemma}
\label{Zpqr}
For $\epsilon\in(0,1)$ and integers $p\geq q\geq 1$ and $r\geq 1$, 
every $\epsilon$-separator of $Z_{p,q,r}$ has size at least $\frac{2(1-\epsilon)pqr}{p+q}\geq (1-\epsilon)qr$. 
\end{lemma}

\begin{proof}
The vertices of $Z_{p,q,r}$ can be partitioned into $pr$ `columns' inducing paths of length $2q$ and $qr$ `rows' inducing paths of length $2p$, such that each row and column are joined by an edge. Let $S$ be an $\epsilon$-separator of $Z_{p,q,r}$. Thus $S$ avoids at least $pr-|S|$ columns and at least $qr-|S|$ rows. Since each row and column are adjacent, the union of these rows and columns that avoid $S$ induces a connected subgraph with at least $2q(pr-|S|)+2p(qr-|S|) = 4pqr -2|S|(p+q)$ vertices. Thus $4pqr -2|S|(p+q) \leq \epsilon|V(Z_{p,q,r})|=4\epsilon pqr$. 
Hence
$|S| \geq \frac{2(1-\epsilon)pqr}{p+q}$, which is at least $(1-\epsilon)qr$ since $p\geq q$. 
\end{proof}

\begin{theorem} 
\label{MapGraphTreewidthLowerBound}
For all $g\geq 0$ and $d\geq 8$, for infinitely many integers $n$, there is an  $n$-vertex $(g,d)$-map graph with treewidth 
$\Omega(\sqrt{(g+1)dn})$ and layered treewidth $\Omega((g+1)d)$.  
\end{theorem}

\begin{proof}
Let $r:= \floor{\frac{d}{4}}$. Thus $r\geq 2$.

First suppose that $g\leq 19$. Infinitely many values of $n$ satisfy $n=4q^2r$ for some integer $q\geq 1$. Let $G$ be $Z_{q,q,r}$. Then $G$ has $n$ vertices. As observed above, $G$ is a $(0,4r)$-map graph and thus a $(g,d)$-map graph. \cref{Zpqr} implies that every $\frac12$-separator of $G$ has size at least $\frac12 qr$. \cref{RS} thus implies that $G$ has treewidth at least $\frac12 qr-1$, which is  $\Omega(\sqrt{(g+1)dn})$, as desired. 

Now assume that $g\geq 20$. By \cref{Expander} there is a 4-regular expander $H$ on $m:=\floor{\frac{g}{4}}\geq 5$ vertices. Thus $H$ has $2m$ edges, $H$ embeds in the orientable surface with $2m$ handles, and thus has Euler genus at most $4m\leq g$. For infinitely many values of $n$, we have that 
$n=(4q^2r -16r)m$ for some integer $q\geq 100$. 

Let $G_0$ be obtained from $H$ as follows. For each vertex $v$ of $H$ introduce a copy of the $(q+1)\times(q+1)$ grid graph with the four corner vertices deleted, denoted by $Y_v$. For each edge $vw$ of $H$, identify one side of $Y_v$ with $Y_w$ (where a side consists of a $(q-1)$-vertex path). The sides are identified according to the embedding of $H$, so that $G_0$ is embedded in the same surface as $H$. Note that each edge of $H$ is associated with a copy of the $(2q+1)\times(q+1)$ grid graph with six vertices deleted in $G_0$. Each face of $G_0$ corresponds to a face of $H$ or is a 4-face inside one of the grid graphs. Now, subdivide each edge $r-1$ times. For each face inside one of the grid graphs, which is now a face $f$ of size $4r$, add a vertex of degree $4r$ adjacent to each vertex on the boundary of $f$. So $G_0$ embeds in the same surface as $H$.  Label the resulting triangular faces of $G_0$ as nations. Label the faces of $G_0$ that correspond to the original faces of $H$ as lakes. Every vertex of $G_0$ is incident to at most $\max\{d,8\}=d$ nations. 

Let $G$ be the $(g,d)$-map graph of $G_0$, as illustrated in \cref{YpqrZpqrEdge}. Each vertex of $H$ is associated in $G$ with a copy of $Z_{q,q,r}$ with the four corner cliques of size $4r$ deleted. Denote this subgraph by $Z^v$, which contains $4q^2r-16r$ vertices in $G$. Each edge $vw$ of $H$ is associated in $G$ with a copy of $Z_{2q,q,r}$ with eight cliques of size $4r$ deleted. Denote this subgraph by $Z^{vw}$, which contains $8q^2r-32r$ vertices in $G$. In total, $G$ has $(4q^2r -16r)m=n$ vertices. 

Let $S$ be a $\frac12$-separator in $G$. Let $A_1,\dots,A_c$ be the components of $G-S$. Thus $|A_i|\leq\frac{1}{2}n$ for $i\in[c]$. Initialise sets $S':=A'_1:=\dots:=A'_c:=\emptyset$. 

\begin{figure}[H]
\centering
\includegraphics[scale=0.5]{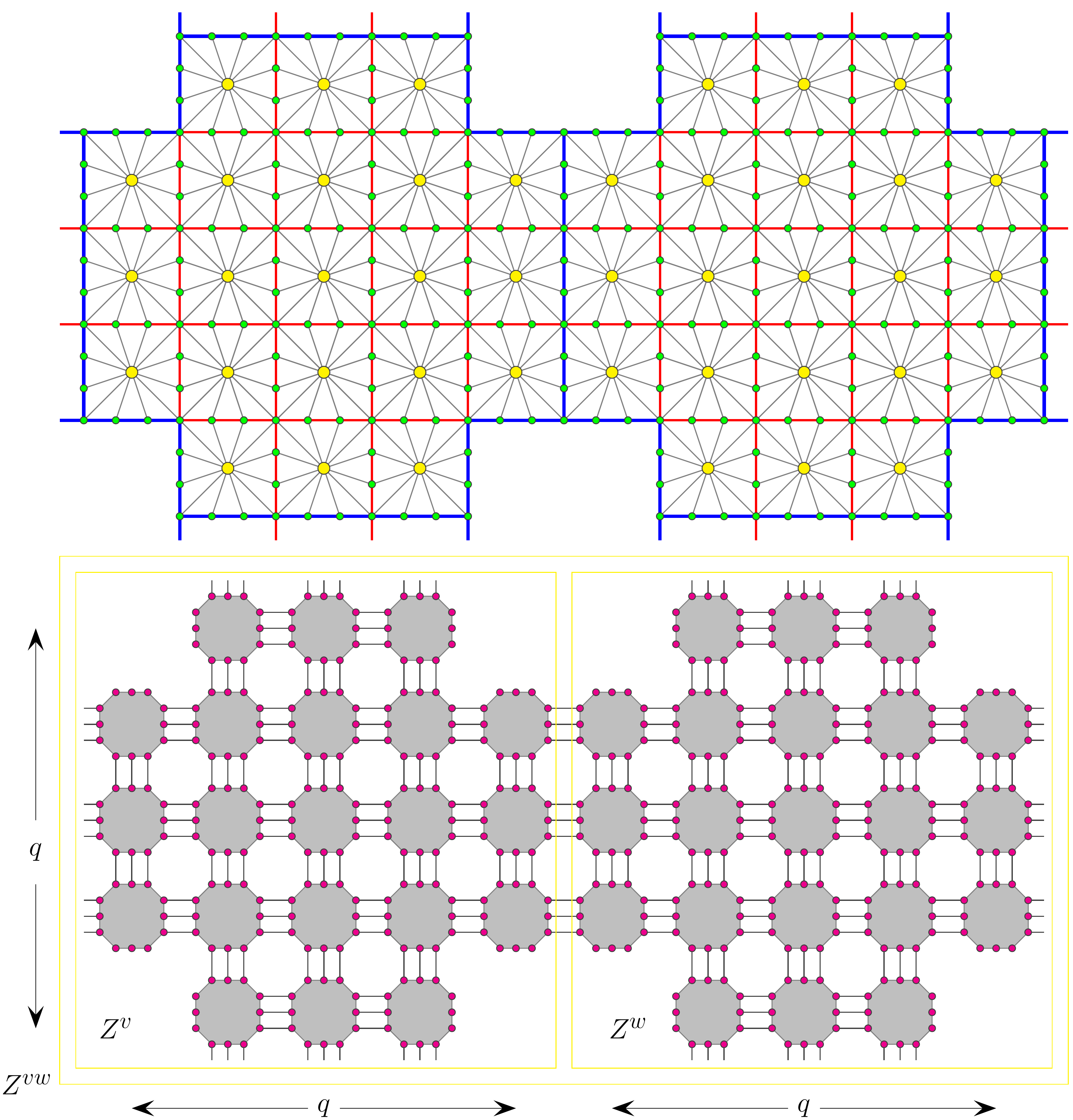}
\caption{A subgraph $Z^{vw}$ of $G$ in the proof of \cref{MapGraphTreewidthLowerBound}.}
\label{YpqrZpqrEdge}
\end{figure}

Consider each vertex $v$ of $H$. If $|S\cap Z^v|\geq\frac{qr}{6}-16r$ then put $v\in S'$. Otherwise, $|S\cap Z^v| < \frac{qr}{6}-16r$. Suppose that $S\cap Z^v$ is a $\frac{5}{6}$-separator of $Z^v$. Then $S\cap Z^v$ plus the $16r$ deleted vertices form a $\frac{5}{6}$-separator in $Z_{q,q,r}$, which has size at least $\frac{qr}{6}$ by \cref{Zpqr}. Thus $|S\cap Z^v|\geq\frac{qr}{6}-16r$, which is a contradiction. Hence $S\cap Z_v$ is not a $\frac{5}{6}$-separator of $Z^v$. Hence some component of $Z^v-S$ has at least $\frac{5}{6} |Z^v|$ vertices. 
Since $\frac{5}{6}>\frac{1}{2}$, exactly one component of $Z^v-S$ has at least $\frac{5}{6} |Z^v|$ vertices. 
This component is a subgraph of $A_i$ for some $i\in[c]$; add $v$ to $A'_i$. Thus $S',A'_1.\dots,A'_c$ is a partition of $V(H)$. 

We now prove that $S'$ is a $\frac{3}{5}$-separator in $H$. Suppose that $v\in A'_i$ and $w\in A'_j$ for some edge $vw$ of $H$. 
Since $v\not\in S'$ and $w\not\in S'$, we have $|S\cap Z^v|<\frac{qr}{6}-16r$ and $|S\cap Z^w|<\frac{qr}{6}-16r$. Thus $|S\cap Z^{vw}| <\frac{qr}{3}-32r$. 

Suppose that $S\cap Z^{vw}$ is a $\frac{3}{4}$-separator of $Z^{vw}$. Then $S\cap Z^{vw}$ plus the $32r$ deleted vertices form a $\frac{3}{4}$-separator in $Z_{2q,q,r}$, which has size at least $\frac{2(1-3/4)(2q)qr}{2q+q}=\frac{qr}{3}$ by \cref{Zpqr} (with $p=2q$). Thus $|S\cap Z^{vw}|\geq\frac{qr}{3}-32r$, which is a contradiction. 
Hence $S\cap Z^{vw}$ is not a $\frac{3}{4}$-separator of $Z^{vw}$. 
Therefore some component $X$ of $Z^{vw}-S$ contains at least $\frac{3}{4}|Z^{vw}|=\frac{3}{2}|Z^v|=\frac{3}{2}|Z^w|$ vertices. 
Of course, each of $Z^v$ and $Z^w$ can contain at most $|Z^v|=|Z^w|$ vertices in $X$. 
Thus $X$ contains at least half the vertices in both $Z^v$ and $Z^w$. 
Hence, by construction, $v$ and $w$ are in the same $A'_i$. 
That is, there is no edge of $H$ between distinct $A'_i$ and $A'_j$, and each component of $H-S'$ is contained in some $A'_i$. 
For each $i\in[c]$, 
$$\half(4q^2r -16r) m = \half n \geq |A_i| \geq \tfrac{5}{6} (4q^2r-16r) |A'_i|.$$ 
Thus $|A'_i| \leq \frac{3}{5}m$. Therefore $S'$ is a $\frac{3}{5}$-separator in $H$. 

By \cref{Expander}, $|S'|\geq\beta m$ for some constant $\beta>0$. 
Thus $|S|\geq (\frac{qr}{6}-16r) |S'|\geq  \beta mr ( \tfrac{q}{6} - 16)$. 
By \cref{RS}, $G$ has treewidth at least 
$$\beta mr ( \tfrac{q}{6} - 16)-1 \geq \Omega( mrq) =
\Omega( \sqrt{ m \cdot r \cdot q^2rm} ) = 
\Omega( \sqrt{ g d n} ),$$
as desired. 

Finally, by \cref{Norine}, if $G$ has layered treewidth $\ell$ then $\Omega(\sqrt{(g+1)dn}) \leq \tw(G) \leq 2\sqrt{\ell n}$, implying $\ell \geq \Omega((g+1)d)$. 
\end{proof}

For $gd\geq n$ the trivial upper bound of $\tw(G)\leq n$ is better than that given in \cref{MapGraphTreewidth}. We conclude that the maximum treewidth of $(g,d)$-map graphs on $n$ vertices is $\Theta(\min\{n,\sqrt{(g+1)(d+1)n}\})$ for arbitrary $g,d,n$. This completes the proof of \cref{MapGraphTreewidthSummary}.


\section{Pathwidth}

It is well known that hereditary graph classes with treewidth $O(n^\epsilon)$, for some fixed $\epsilon\in(0,1)$, in fact have pathwidth $O(n^\epsilon)$; see \citep{Bodlaender-TCS98} for example. In particular, the following more specific result means that all the $O(\sqrt{n})$ treewidth upper bounds in this paper lead to $O(\sqrt{n})$ pathwidth upper bounds. We include the proof for completeness. 

\begin{lemma}
\label{PathwidthLemma}
Let $G$ be a graph with $n$ vertices such that every induced subgraph $G'$ of $G$ with $n'$ vertices has treewidth at most $c\sqrt{n'}-1$ for some constant $c\geq (1-\sqrt{2/3})^{-1}$. Let $c':=c(1-\sqrt{2/3})^{-1}$. Then
$$\pw(G) \leq c'\sqrt{n}-1 <\frac{11c}{2}\sqrt{n}-1.$$
\end{lemma}

\begin{proof}
We proceed by induction on $n'\geq 1$ with the hypothesis that every non-empty subgraph $G'$ of $G$ with $n'$ vertices has pathwidth at most $c'\sqrt{n'}-1$. If $n'=1$ then $G'$ has pathwidth $0$ and the claim holds since $c\geq (1-\sqrt{2/3})^{-1}$. Consider a subgraph $G'$ of $G$ with $n'$ vertices. By assumption, $G'$ has treewidth at most $c\sqrt{n'}-1$. By \cref{RS}, $G'$ has a $\frac12$-separator $S$ of size at most $c\sqrt{n'}$. Thus each component of $G'-S$ contains at most $\frac{n'}{2}$ vertices. Group the components of $G'-S$ as follows, starting with each component in its own group. So initially each group has at most $\frac{n'}{2}\leq \frac{2}{3}n'$ vertices. While there are at least three groups, merge the two smallest groups, which have at most $\frac{2}{3}n'$ vertices in total. Upon termination, there are at most two groups, each with at most $\frac{2}{3}n'$ vertices. Let $A$ and $B$ be the subgraphs of $G'$ induced by the two groups. By induction, $A$ and $B$ each have pathwidth at most $c'\sqrt{\frac{2}{3}n'}-1$. Let $A_1,\dots,A_a$ and $B_1,\dots,B_b$ be the corresponding path decompositions of $A$ and $B$ respectively. Then $A_1\cup S,\dots,A_a\cup S,S,B_1\cup S,\dots,B_b\cup S$ is a path decomposition of $G'$ with width 
$$c'\sqrt{\frac{2}{3}n'}-1+|S| \leq c'\sqrt{\frac{2}{3}n'}-1+c\sqrt{n'} = (c'\sqrt{\frac{2}{3}}+c)\sqrt{n'}-1 = c'\sqrt{n'}-1,$$ 
as desired. Hence $G$ has pathwidth at most $c'\sqrt{n}$. 
\end{proof}

\cref{Norine} and \cref{PathwidthLemma} imply:

\begin{theorem}
\label{LayeredTreewidthPathwidth}
Every $n$-vertex graph with layered treewidth $k$ has pathwidth at most $11\sqrt{kn}-1$.
\end{theorem}

\paragraph{Acknowledgement.} 
This research was initiated at the 3rd Workshop on Graphs and Geometry held at the Bellairs Research Institute in 2015.  Thanks to Thomas Bl\"asius for pointing out a minor error in a preliminary version of this paper. 



\begin{thebibliography}{29}
\expandafter\ifx\csname natexlab\endcsname\relax\def\natexlab#1{#1}\fi

\bibitem[{Bodlaender(1998)}]{Bodlaender-TCS98}
\textsc{H.~L. Bodlaender}, A partial $k$-arboretum of graphs with bounded
  treewidth. \emph{Theoret. Comput. Sci.}, \textbf{209(1-2)}:1--45, 1998.

\bibitem[{Chen(2001)}]{Chen01}
\textsc{Z.-Z. Chen}, Approximation algorithms for independent sets in map
  graphs. \emph{J. Algorithms}, \textbf{41(1)}:20--40, 2001.

\bibitem[{Chen(2007)}]{Chen-JGT07}
\textsc{Z.-Z. Chen}, New bounds on the edge number of a {$k$}-map graph.
  \emph{J. Graph Theory}, \textbf{55(4)}:267--290, 2007.

\bibitem[{Chen \emph{et~al.}(2002)Chen, Grigni, and Papadimitriou}]{CGP02}
\textsc{Z.-Z. Chen, M.~Grigni, and C.~H. Papadimitriou}, Map graphs. \emph{J.
  ACM}, \textbf{49(2)}:127--138, 2002.

\bibitem[{Demaine \emph{et~al.}(2004/05)Demaine, Fomin, Hajiaghayi, and
  Thilikos}]{DH-SJDM04}
\textsc{E.~D. Demaine, F.~V. Fomin, M.~Hajiaghayi, and D.~M. Thilikos},
  Bidimensional parameters and local treewidth. \emph{SIAM J. Discrete Math.},
  \textbf{18(3)}:501--511, 2004/05.

\bibitem[{Demaine \emph{et~al.}(2005)Demaine, Fomin, Hajiaghayi, and
  Thilikos}]{DFHT05}
\textsc{E.~D. Demaine, F.~V. Fomin, M.~Hajiaghayi, and D.~M. Thilikos},
  Fixed-parameter algorithms for {$(k,r)$}-center in planar graphs and map
  graphs. \emph{ACM Trans. Algorithms}, \textbf{1(1)}:33--47, 2005.

\bibitem[{Demaine and Hajiaghayi(2004)}]{DH-SODA04}
\textsc{E.~D. Demaine and M.~Hajiaghayi}, Equivalence of local treewidth and
  linear local treewidth and its algorithmic applications. In \emph{Proc. 15th
  Annual ACM-SIAM Symposium on Discrete Algorithms (SODA '04)}, pp. 840--849,
  SIAM, 2004.

\bibitem[{Dujmovi{\'c}(2015)}]{Duj15}
\textsc{V.~Dujmovi{\'c}}, Graph layouts via layered separators. \emph{J.
  Combin. Theory Series B.}, \textbf{110}:79--89, 2015.

\bibitem[{Dujmovi{\'c} \emph{et~al.}(2005)Dujmovi{\'c}, Morin, and
  Wood}]{DMW05}
\textsc{V.~Dujmovi{\'c}, P.~Morin, and D.~R. Wood}, Layout of graphs with
  bounded tree-width. \emph{SIAM J. Comput.}, \textbf{34(3)}:553--579, 2005.

\bibitem[{Dujmovi{\'c} \emph{et~al.}(2013)Dujmovi{\'c}, Morin, and
  Wood}]{DMW13}
\textsc{V.~Dujmovi{\'c}, P.~Morin, and D.~R. Wood}, {Layered separators in
  minor-closed families with applications}. Electronic preprint
  \arXiv{1306.1595}, 2013.

\bibitem[{Dujmovi{\'c} \emph{et~al.}(2015)Dujmovi{\'c}, Sidiropoulos, and
  Wood}]{DSW16}
\textsc{V.~Dujmovi{\'c}, A.~Sidiropoulos, and D.~R. Wood}, Layouts of expander
  graphs. \emph{Chicago J. of Theoretical Computer Science}, \textbf{2016(1)},
  2015.

\bibitem[{Dvor{\'a}k and Norin(2014)}]{DN14}
\textsc{Z.~Dvor{\'a}k and S.~Norin}, {Treewidth of graphs with balanced
  separations}. Electronic preprint \arXiv{1408.3869}, 2014.

\bibitem[{Eppstein(2000)}]{Epp-Algo-00}
\textsc{D.~Eppstein}, {Diameter and treewidth in minor-closed graph families}.
  \emph{Algorithmica}, \textbf{27}:275{--}291, 2000.

\bibitem[{Fomin \emph{et~al.}(2012)Fomin, Lokshtanov, and Saurabh}]{FLS-SODA12}
\textsc{F.~V. Fomin, D.~Lokshtanov, and S.~Saurabh}, Bidimensionality and
  geometric graphs. In \emph{Proceedings of the {T}wenty-{T}hird {A}nnual
  {ACM}-{SIAM} {S}ymposium on {D}iscrete {A}lgorithms}, pp. 1563--1575, 2012.

\bibitem[{Gilbert \emph{et~al.}(1984)Gilbert, Hutchinson, and
  Tarjan}]{GHT-JAlg84}
\textsc{J.~R. Gilbert, J.~P. Hutchinson, and R.~E. Tarjan}, {A separator
  theorem for graphs of bounded genus}. \emph{J. Algorithms},
  \textbf{5(3)}:391{--}407, 1984.

\bibitem[{Grigoriev and Bodlaender(2007)}]{GB07}
\textsc{A.~Grigoriev and H.~L. Bodlaender}, {Algorithms for graphs embeddable
  with few crossings per edge}. \emph{Algorithmica}, \textbf{49(1)}:1{--}11,
  2007.

\bibitem[{Grohe(2003)}]{Grohe-Comb03}
\textsc{M.~Grohe}, Local tree-width, excluded minors, and approximation
  algorithms. \emph{Combinatorica}, \textbf{23(4)}:613--632, 2003.

\bibitem[{Grohe and Marx(2009)}]{GM-JCTB}
\textsc{M.~Grohe and D.~Marx}, {On tree width, bramble size, and expansion}.
  \emph{J. Combin. Theory Ser. B}, \textbf{99(1)}:218{--}228, 2009.

\bibitem[{Guy \emph{et~al.}(1968)Guy, Jenkyns, and Schaer}]{GJS68}
\textsc{R.~K. Guy, T.~Jenkyns, and J.~Schaer}, {The toroidal crossing number of
  the complete graph}. \emph{J. Combinatorial Theory}, \textbf{4}:376{--}390,
  1968.

\bibitem[{Halin(1976)}]{Halin76}
\textsc{R.~Halin}, {$S$-functions for graphs}. \emph{J. Geometry},
  \textbf{8(1-2)}:171{--}186, 1976.

\bibitem[{Hoory \emph{et~al.}(2006)Hoory, Linial, and Wigderson}]{HLW06}
\textsc{S.~Hoory, N.~Linial, and A.~Wigderson}, {Expander graphs and their
  applications}. \emph{Bull. Amer. Math. Soc. (N.S.)},
  \textbf{43(4)}:439{--}561, 2006.

\bibitem[{Leighton and Rao(1999)}]{LR99}
\textsc{T.~Leighton and S.~Rao}, {Multicommodity max-flow min-cut theorems and
  their use in designing approximation algorithms}. \emph{J. ACM},
  \textbf{46(6)}:787{--}832, 1999.

\bibitem[{Otachi and Suda(2011)}]{OS11}
\textsc{Y.~Otachi and R.~Suda}, Bandwidth and pathwidth of three-dimensional
  grids. \emph{Discrete Math.}, \textbf{311(10-11)}:881--887, 2011.

\bibitem[{Pach and T{\'o}th(1997)}]{PachToth-Comb97}
\textsc{J.~Pach and G.~T{\'o}th}, Graphs drawn with few crossings per edge.
  \emph{Combinatorica}, \textbf{17(3)}:427--439, 1997.

\bibitem[{Reed(1997)}]{Reed97}
\textsc{B.~A. Reed}, {Tree width and tangles: a new connectivity measure and
  some applications}. In \emph{Surveys in combinatorics}, vol. 241 of
  \emph{London Math. Soc. Lecture Note Ser.}, pp. 87{--}162, Cambridge Univ.
  Press, 1997.

\bibitem[{Robertson and Seymour(1986)}]{RS-GraphMinorsII-JAlg86}
\textsc{N.~Robertson and P.~D. Seymour}, {Graph minors. II. Algorithmic aspects
  of tree-width}. \emph{J. Algorithms}, \textbf{7(3)}:309{--}322, 1986.

\bibitem[{Schaefer(2014)}]{Schaefer14}
\textsc{M.~Schaefer}, {The graph crossing number and its variants: A survey}.
  \emph{Electron. J. Combin.}, \textbf{DS21}, 2014.

\bibitem[{Shahrokhi(2013)}]{Shahrokhi13}
\textsc{F.~Shahrokhi}, New representation results for planar graphs. In
  \emph{29th European Workshop on Computational Geometry (EuroCG 2013)}, pp.
  177--180, 2013, \arXiv{1502.06175}.

\bibitem[{Shahrokhi \emph{et~al.}(1996)Shahrokhi, Sz{\'e}kely, S{\'y}kora, and
  Vrt'o}]{SSSV96}
\textsc{F.~Shahrokhi, L.~A. Sz{\'e}kely, O.~S{\'y}kora, and I.~Vrt'o},
  {Drawings of graphs on surfaces with few crossings}. \emph{Algorithmica},
  \textbf{16(1)}:118{--}131, 1996.

\end{thebibliography}

\end{document}